\documentclass[10.5pt,letterpaper]{amsart}

\usepackage[letterpaper,margin=1.2in]{geometry}

\usepackage{amsmath,amssymb,amsthm}
\usepackage{parskip}
\usepackage{enumerate}
\usepackage[retainorgcmds]{IEEEtrantools}
\usepackage{hyperref}
\usepackage{mathtools,pifont}
\usepackage{tikz-cd}
\usepackage{tikz}
\usepackage{mathrsfs}
\usepackage{xcolor}

\newtheorem{theorem}{Theorem}[section]

\newtheorem{lemma}[theorem]{Lemma}
\newtheorem{proposition}[theorem]{Proposition}
\newtheorem{corollary}[theorem]{Corollary}
\newtheorem{fact}[theorem]{Fact}

\newtheorem*{theorem*}{Theorem}

\numberwithin{equation}{section}

\theoremstyle{definition}
\newtheorem{definition}[theorem]{Definition}
\newtheorem{notation}[theorem]{Notation}

\theoremstyle{remark}
\newtheorem{remark}[theorem]{Remark}

\newcommand\R{\mathbb{R}}
\newcommand\C{\mathbb{C}}

\newcommand\N{\mathbb{N}}

\newcommand\E{\mathbb{E}}

\newcommand{\cB}{\mathcal{B}}
\newcommand{\cM}{\mathcal{M}}
\newcommand{\cN}{\mathcal{N}}
\newcommand{\cO}{\mathcal{O}}

\newcommand{\cU}{\mathcal{U}}

\newcommand{\bX}{\mathbf{X}}
\newcommand{\bY}{\mathbf{Y}}
\newcommand{\bZ}{\mathbf{Z}}
\newcommand{\bS}{\mathbf{S}}
\newcommand{\bW}{\mathbf{W}}
\newcommand{\bR}{\mathbf{R}}

\newcommand{\scal}[1]{\left\langle #1 \right\rangle}

\DeclareMathOperator{\Tr}{Tr}
\DeclareMathOperator{\tr}{tr}

\DeclareMathOperator{\sa}{sa}
\DeclareMathOperator{\op}{op}

\DeclarePairedDelimiter{\norm}{\lVert}{\rVert}
\DeclarePairedDelimiter{\ip}{\langle}{\rangle}

\title{An Elementary Proof of the Inequality $\chi \leq \chi^*$ for Conditional Free Entropy}
\author{David Jekel and Jennifer Pi}

\address{Department of Mathematics\\University of California, Irvine, 410 Rowland Hall (Bldg.\# 400),
	Irvine, CA 92697-3875}
\email{jspi@math.uci.edu}
\urladdr{https://sites.uci.edu/jpi314/}

\address{Department of Mathematics\\University of California, San Diego, 9500 Gilman Dr,
	La Jolla, CA 92093}
\email{djekel@ucsd.edu}
\urladdr{https://davidjekel.com/}

\date{}

\begin{document}
	
	\begin{abstract}
		Through the study of large deviations theory for matrix Brownian motion, Biane-Capitaine-Guionnet proved the inequality $\chi(X) \leq \chi^*(X)$ that relates two analogs of entropy in free probability defined by Voiculescu. We give a new proof of $\chi \leq \chi^*$ that is elementary in the sense that it does not rely on stochastic differential equations and large deviation theory.  Moreover, we generalize the result to conditional microstates and non-microstates free entropy.
	\end{abstract}
	
	\maketitle
	
	% \tableofcontents
	
	% \newpage
	
	\section{Introduction}
	
	Voiculescu developed free probability as a non-commutative analog to classical probability that deals with free products rather than tensor products, and subsequent connections to operator algebras and random matrices have been numerous (see, e.g. \cite[Chapter 4 \& 6]{MingoSpeicher}).  While pursuing the analogy between classical and free probability, Voiculescu developed several candidates for the analog of (continuous) entropy for non-commutative random variables $\bX$ \cite{VoiculescuFE1,VoiculescuFE2,VoiculescuFE5}. The first candidate is \textit{microstates free entropy}, $\chi(\bX)$, which takes its inspiration from Boltzmann's characterization of entropy via the counting of microstates that may make up a single macrostate \cite{Voiculescu2002}. The second candidate, \textit{non-microstates free entropy}, $\chi^*(\bX)$, is developed in direct analog to the classical relationship between entropy and Fisher information \cite{VoiculescuFE5}.  Microstates free entropy has spawned several variants (such as the $1$-bounded entropy of Jung \cite{Jung2007S1B} and Hayes \cite{Hayes2018}) which have had numerous applications to von Neumann algebras, see e.g.\ \cite{CIKE,DykemaFreeEntropy, GePrime, PopaGeThin, geshen, HJK2022,HJNS2021,Shlyakhtenko2002,Shlyakhtenko2015, VoiculescuFE3}.
	
	A major open question in the study of free entropy is the unification problem: are the microstates and non-microstates notions of entropy the same?  The recently announced resolution of the Connes embedding problem (see \cite{Goldbring2021connes, MIP*eqsRE}) implies that there are non-commutative random variables $X$ with $\chi(X) = -\infty$ and $\chi^*(X) > -\infty$, but the unification problem is still largely open in the case of Connes-embeddable non-commutative random variables (i.e.\ those that admit \emph{any} matrix approximations).  Biane-Capitaine-Guionnet \cite{BCG2003} proved in 2003 that the microstates free entropy is always bounded above by the non-microstates free entropy as a consequence of their study of large-deviation principles for matrix Brownian motion.  In 2017, Dabrowski showed that $\chi$ and $\chi^*$ agree for free Gibbs laws with convex potentials that satisfy certain regularity and growth conditions \cite{Dabrowski2017}.  Another proof of this result was given by the first author in \cite{JekelEntropy}, which was more elementary in the sense that it did not rely as much on stochastic differential equations and ultraproducts.
	
	In this paper, we will revisit the $\chi \leq \chi^*$ result of Biane, Capitaine, and Guionnet.  Their work showed a lot more than simply $\chi \leq \chi^*$, since they produced a large deviation upper bound for matrix Brownian motion in the large $N$ limit.  However, as the intuitive sketch by Shlyakhtenko suggests \cite[\S 4.6]{ShlyakhtenkoParkCity} and we will prove, the result $\chi \leq \chi^*$ does not require this machinery.  Rather, in the spirit of \cite{JekelEntropy}, we make the analogy between Fisher information in classical entropy and free Fisher information in free entropy into precise limiting results.  A key observation from \cite[Proposition B.7]{ST2022} connecting free and classical entropy is that microstates free entropy $\chi(\bX)$ can be realized as the supremum of subsequential limits of normalized classical entropy $h^{(n)}(\bX^{(n)})$ of all random matrix models $\bX^{(n)}$ for $\bX$ satisfying certain conditions (this is also reminiscent of \cite{BD2013}).  In this paper, we will give a similar estimate relating free Fisher information and the large-$n$ limit of classical Fisher information, and with the combination of these tools we will show $\chi \leq \chi^*$.
	
	Furthermore, we generalize this to the conditional setting, comparing Voiculescu's non-microstates entropy of $\bX$ conditioned on a subalgebra $\cB$ from \cite{VoiculescuFE5} with Shlyakhtenko's free microstates entropy for $\bX$ conditioned on $\cB$, called $\chi(\bX \mid \cB)$ \cite{Shlyakhtenko2002}.  Similar to the non-conditional case, one key ingredient is that Shlyakhtenko's conditional microstates entropy can analogously be realized as the supremum of subsequential limits of normalized conditional classical entropy $h^{(n)}(\bX^{(n)} \mid \bY^{(n)})$ over certain random matrix models $(\bX^{(n)},\bY^{(n)})$ for $(\bX,\bY)$, where $\bY$ is a tuple of generators for $\cB$ (see Theorem \ref{thm: rand mx conditional entropy 2}).
	
	Our main results are as follows.  We phrase some of our results in terms of ultrafilter versions of $\chi$ in order to highlight the connection between entropy and embeddings into ultraproducts (see \S \ref{sec: conditional microstates}) as well as to avoid having to pass to subsequences. For an $m$-tuple $\bX$ in a tracial von Neumann algebra and a subalgebra $\cB$ generated by a tuple $\bY$, we show that
	\begin{itemize}
		\item The conditional microstates free entropy $\chi^{\cU}(\bX \mid \cB)$ is the supremum of the ultralimits of the normalized classical conditional entropy $h^{(n)}(\bX^{(n)} \mid \bY^{(n)})$ over random matrix models $(\bX^{(n)},\bY^{(n)})$ of $(\bX,\bY)$ satisfying certain growth bounds. (Theorem \ref{thm: rand mx conditional entropy 2})
		\item Shlyakhtenko's conditional entropy $\overline{\chi}(\mathbf{X} \mid \cB)$ is the supremum of $\chi^{\cU}(\mathbf{X} \mid \mathcal{B})$ over choices of ultrafilter $\cU$ (Lemma \ref{lem: Shl-sup}).
		\item Given a sequence of deterministic microstates for $\bY$, the conditional free Fisher information is bounded above by a limit of the classical Fisher information of any sequence of microstates $\bX^{(n)}$ for $\bX$ satisfying some moment boundedness conditions. (Proposition \ref{prop: cond Phi* upper bound})
		\item We prove that $\overline{\chi}(\bX \mid \cB) \leq \chi^*(\bX \mid \cB)$, where $\chi^*(\bX: \cB)$ is Voiculescu's non-microstates free entropy relative to $\cB$. (Theorem \ref{thm: cond main}).
	\end{itemize}
	The paper is organized as follows:
	\begin{itemize}
		\item \S \ref{sec: prelim} reviews background on operator algebras, free probability, and ultraproducts.
		\item \S \ref{sec: classical entropy} explains classical entropy and Fisher information as background and motivation.
		\item \S \ref{sec: conditional microstates} describes several versions of conditional microstates entropy and establishes their relationships with each other and with classical conditional entropy.
		\item \S \ref{sec: conditional non-microstates} reviews Voiculescu's conditional free Fisher information, and gives a characterization of it that does not explicitly reference the free score functions, which will be used for the main proof.
		\item \S \ref{sec: main proof} provides the proof of the main theorm : for any tracial von Neumann algebra $(\cM, \tau)$, any $\mathrm{W}^*$-subalgebra $\cB$, and any $m$-tuple of self-adjoint random variables $\bX$ from $\cM$, we have $\overline{\chi}(\bX \mid \cB) \leq \chi^*(\bX: \cB)$.
	\end{itemize}
	
	\subsection*{Acknowledgements}
	
	DJ was supported by an NSF postdoctoral fellowship (DMS-2002826), the Fields Institute, and Ilijas Farah's NSERC grant.
	
	JP was partially supported by the NSF, grant DMS-2054477.
	
	DJ thanks Dima Shlyakhtenko and Yoann Dabrowski for past discussions about free entropy.  We also thank the referee for numerous comments that improved the correctness and clarity of the paper.
	
	\section{Background on operator algebras and free probability} \label{sec: prelim}
	
	\subsection{Tracial von Neumann algebras}
	
	We assume some familiarity with tracial von Neumann algebras; some standard references for background on these objects are \cite{AP2017, KadisonRingroseI, Sakai1971, TakesakiI}. Concise introductions can be found in \cite{Ioana2023} and \cite[\S 2]{GJNS2022}.  We record some basic definitions and notations in this section for usage in this paper.
	
	First, we recall the notion of $\mathrm{C}^*$-algebras.
	\begin{enumerate}[(1)]
		%		\item A (unital) algebra over $\C$ is a unital ring $A$ with a unital inclusion map $\C \to A$.
		\item A (unital) $*$-algebra is an algebra $A$ equipped with a conjugate linear involution $*$ such that $(ab)^* = b^* a^*$.
		\item A unital $\mathrm{C}^*$-algebra is a $*$-algebra $A$ equipped with a complete norm $\norm{\cdot}$ such that $\norm{ab} \leq \norm{a} \norm{b}$ and $\norm{a^*a} = \norm{a}^2$ for $a, b \in A$.
	\end{enumerate}
	
	A collection of fundamental results in $\mathrm{C}^*$-algebra theory establishes that $\mathrm{C}^*$-algebras can always be represented as algebras of operators on Hilbert spaces.  If $H$ is a Hilbert space, the algebra of bounded operators $B(H)$ is a $\mathrm{C}^*$-algebra.  Conversely, every unital $\mathrm{C}^*$-algebra can be embedded into $B(H)$ by some unital and isometric $*$-homomorphism $\rho$.  
	%By isometric, we mean that $\norm{\rho(a)} = \norm{a}$, where $\norm{\rho(a)}$ is the operator norm on $B(H)$ and $\norm{a}$ is the given norm on the $\mathrm{C}^*$-algebra $A$.
	
	A \emph{von Neumann algebra} $\cM$ is a unital $\mathrm{C}^*$-algebra represented on some $B(H)$ that is also closed in the strong operator topology (equivalently in the weak operator topology) coming from $B(H)$.
	If $\cM$ has a linear functional $\tau: \cM \rightarrow \C$ which is:
	\begin{itemize}
		\item positive: $\tau(X^*X) \geq 0$ for all $X \in \cM$;
		\item unital: $\tau(1) = 1$;
		\item faithful: $\tau(X^*X) = 0$ if and only if $X = 0$;
		\item normal: $\tau: \cM \rightarrow \C$ is continuous in the weak$-\ast$ toplogy;
		\item tracial: $\tau(XY) = \tau(YX)$ for all $X, Y \in \cM$,
	\end{itemize}
	then we call $\tau$ a trace on $\cM$. We say that $(\cM, \tau)$ is a tracial von Neumann algebra. Also, when the trace is implicit or clear, we often drop $\tau$ from the notation and just refer to $\cM$ as a tracial von Neumann algebra.  We write $\cM_{\sa}$ for the collection of self-adjoint elements of $\cM$.

	For an element $X \in \cM$, we use the $2$-norm from von Neumann algebras: $\norm{X}_2 = \tau(X^*X)^{1/2}$. 
	Note in particular that for a matrix $X \in M_n(\C)$, this means
	\[
	\norm{X}_2 = \tr_n(X^*X)^{1/2} = \left( \frac{1}{n} \sum_{i,j=1}^n |X_{i,j}|^2 \right)^{1/2}.
	\]
	More generally, for $p \in [1,\infty)$ the non-commutative $p$-norm is given by
	\[
	\norm{X}_p = \tau(|X|^p)^{1/p},
	\]
	where $|X| = (X^*X)^{1/2}$ and $|X|^p$ is defined by continuous functional calculus.  We write ${\norm{X}_\infty = \norm{X}_{\op}}$ for the operator norm of $X \in \cM \subseteq \mathcal{B}(H)$.  These non-commutative $p$-norms satisfy an analog of H\"older's inequality; for proof, see \cite[Theorem 6]{Dixmier1953}, \cite[Theorem 2.1.5]{daSilva2018}.
	
	\begin{fact}[Non-commutative H\"older's inequality] \label{fact: Holder}
		Let $\cM$ be a tracial von Neumann algebra and let $X_1$, \dots, $X_k \in \cM$.  Let $p, p_1$, \dots, $p_k \in [1,\infty]$ with $1/p = 1/p_1 + \dots + 1/p_k$.  Then
		\[
		\norm{X_1 \cdots X_k}_p \leq \norm{X_1}_{p_1} \dots \norm{X_k}_{p_k}.
		\]
	\end{fact}
	
	Given $\cM$ a tracial von Neumann algebra, we can form a Hilbert space $L^2(\cM)$ by taking the completion of $\cM$ with respect to the inner product $\scal{x,y}_\tau = \tau(x^*y)$, which gives rise to the norm $\norm{\cdot}_2$. Moreover, there is a $*$-homomorphism $\cM \to B(L^2(\cM))$ which maps $x$ to the operator of left multiplication by $x$.  For further information on this space, we refer the reader to \cite[\S 2.6]{AP2017}.
	% We also denote by $L^1(\cM)$ the completion of $\cM$ with respect to $\norm{x}_1 = \tau(|x|)$, where $|x| = (x^*x)^{1/2}$.
	
	If $\cM$ is a tracial von Neumann algebra and $\mathbf{X} = (X_i)_{i \in I}$ is a self-adjoint tuple in $\cM$, we denote by $\mathrm{W}^*(\mathbf{X})$ the smallest von Neumann subalgebra of $\cM$ that contains $\mathbf{X}$.  This is called the \emph{von Neumann subalgebra generated by $\mathbf{X}$}.  It is easy to see that $\mathrm{W}^*(\mathbf{X})$ is the closure in the weak-operator topology of the $*$-algebra generated by $\mathbf{X}$.  In fact, every element in $\mathrm{W}^*(\mathbf{X})$ can be approximated by $*$-polynomials in $\mathbf{X}$ in a rather strong sense, a fact we will need to use for Lemma \ref{lemma: cond microstates f.e. embedding}.
	
	Let $\C \scal{x_i : i \in I}$ be the $*$-algebra of non-commutative polynomials in formal self-adjoint variables $x_i$ for $i$ in the index set $I$.  Note that for every tracial von Neumann algebra $\cM$ and self-adjoint $\mathbf{X} = (X_i)_{i \in I}$ in $\cM$, there is a unique $*$-homomorphism $\C \scal{x_i: i \in I} \to \cM$ given by $x_i \mapsto X_i$, which we call \emph{evaluation at $\mathbf{X}$} and denote $p \mapsto p(\mathbf{X})$.
	
	The following lemma is a combination of well-
	known results; for instance, the arguments in \cite[Theorem 2.15]{Shlyakhtenko2002} and \cite[Lemma A.8]{Hayes2018} rely on this result without explaining it in detail.  It is also a key part of the proof of \cite[Theorem 3.30, Proposition 3.32]{JekelCovering}.
	
	\begin{lemma} \label{lem: polynomial approximation}
		Let $\cM$ be a tracial von Neumann algebra, let $\mathbf{R} = (R_i)_{i \in I} \in (0,\infty)^I$, and let $\mathbf{X} = (X_i)_{i \in I}$ a self-adjoint tuple with $\norm{X_i} \leq R_i$.  Let $Z \in \mathrm{W}^*(\mathbf{X})$.  Then for every $\epsilon > 0$, there exists a $p \in \C\ip{x_i: i \in I}$ such that
		\[
		\norm{Z - p(\mathbf{X})}_2 < \epsilon,
		\]
		and for all tracial von Neumann algebras $\cN$ and all $\mathbf{Y}$ with $\norm{Y_i} \leq R_i$, we have
		\[
		\norm{p(\mathbf{Y})}_{\op} \leq \norm{Z}_{\op}.
		\]
	\end{lemma}
	
	\begin{proof}
		Let $\mathrm{C}^*(\mathbf{X}) \subseteq \cM$ be the $\mathrm{C}^*$-algebra generated by $\mathbf{X}$, that is, the operator norm closure of the $*$-algebra generated by $\mathbf{X}$. 
		By the Kaplansky density theorem (see for instance \cite[Theorem 5.3.5]{KadisonRingroseI} or \cite{Kaplansky1953}), the ball of radius $\norm{Z}_{\op}$ in $\mathrm{C}^*(\bX)$ generated by $\bX$ is dense in the ball of radius $\norm{Z}_{\op}$ in $\mathrm{W}^*(\bX)$ with respect to the strong operator topology.  Since approximation in the strong operator topology implies approximation in the $2$-norm associated to the trace, it follows that there exists $Z' \in \mathrm{C}^*(\bX)$ such that $\norm{Z'}_{\op} \leq \norm{Z}_{\op}$ and $\norm{Z - Y}_2 < \epsilon / 2$.
		
		Next, we want to approximate $Z'$ by $p(\mathbf{X})$ for some non-commutative polynomial $p$, but the challenge is to guarantee that $\norm{p(\mathbf{Y})}_{\op} \leq \norm{Z}_{\op}$ for $\mathbf{Y}$ in \emph{any} tracial von Neumann algebra $\cM$.  To this end, we will complete the polynomial algebra to a certain ``universal'' $\mathrm{C}^*$-algebra.  For $*$-polynomials $p$ in infinitely many variables $(x_i)_{i \in I}$, let
		\[
		\norm{p}_u = \sup \left\{ \norm{p(\bY)}_{\op}: \cN \text{ tracial von Neumann algebra}, \bY \in \cN_{\sa}^I, \norm{Y_j} \leq R_i \text{ for } i \in I \right\}.
		\]
		This defines a $\mathrm{C}^*$-norm on $\C\ip{x_i, x_i^*: i \in I}$.  Let $A$ be the completion of $\C\ip{x_i, x_i^*: i \in I}$ into a $\mathrm{C}^*$-algebra.  If $\cN$ is a tracial von Neumann algebra and $\bY \in \cN_{\sa}^I$, then $\norm{p(\bY)}_{\op} \leq \norm{p}_u$ by definition, so there is a $*$-homomorphism $\pi: A \to \mathrm{C}^*(\bX)$ mapping $x_i \in A$ to $X_i \in \cM$ for each $i \in I$.  The image of this homomorphism is dense, and therefore it is surjective because the image of a $\mathrm{C}^*$-algebra under a $*$-homomorphism is closed \cite[II.5.1.2]{Blackadar2006}.  Moreover, by \cite[II.5.1.5]{Blackadar2006}, there exists $a \in A$ such that $\pi(a) = Z'$ and $\norm{a}_A = \norm{Z'}_{\op} \leq \norm{Z}_{\op}$.  By definition of $A$, there exists $p \in \C\scal{x_i: i \in I}$ such that $\norm{p - a}_A < \epsilon/2$ and $\norm{p}_A = \norm{p}_u \leq \norm{a}_A \leq \norm{Z}_{\op}$.  It follows that $\norm{p(\bX) - Z'}_2 \leq \norm{p(\bX) - Z'}_{\op} \leq \norm{\pi(p) - \pi(a)}_{\op} < \epsilon / 2$.  Hence, $\norm{p(\bX) - Z}_2 < \epsilon$ and $\norm{p}_u \leq \norm{Z}_{\op}$, as desired.
	\end{proof}
	
	One more fact needed later for Lemma \ref{lemma: cond microstates f.e. embedding} is the following.
	
	\begin{fact} \label{fact: polynomial uniform continuity}
		Let $p \in \C\scal{x_i: i \in I}$, and let $\mathbf{R} \in (0,\infty)^I$.  Let $\epsilon > 0$.  Then there exists a finite $F \subseteq I$ and $\delta > 0$ such that for all tracial von Neumann algebras $\cN$ and all $\bX$, $\bY \in \cM^I$ with $\max(\norm{X_i},\norm{Y_i}) \leq R_i$ for $i \in I$ and $\norm{X_i - Y_i}_2 < \delta$ for $i \in F$, we have $\norm{p(\bX) - p(\bY)}_2 < \epsilon$.
	\end{fact}
	
	We leave the proof as an exercise.  It suffices to check the claim for monomials, and this can be done using the non-commutative H\"older's inequality and triangle inequality.
	
	\subsection{Non-commutative laws, free independence, and random matrices} \label{subsec: nc laws}
	
	A non-commutative analog of the probability distribution of a random variable is the \emph{non-commutative law} described as follows.
	
	\begin{definition}
		Let $\C \scal{x_i : i \in I}$ be the $*$-algebra of non-commutative polynomials in formal self-adjoint variables $x_i$ for $i$ in the index set $I$ (usually, we take $I$ to be $\{1,\dots,n\}$ or $\N$). For any positive numbers $\mathbf{R} = (R_i)_{i \in I}$, we define the \textit{space of non-commutative laws} in $I$ variables bounded by $\mathbf{R}$, denoted $\Sigma_{I,\mathbf{R}}$, as the space of tracial positive linear functionals
		${\lambda: \C \scal{x_i : i \in I} \rightarrow \C}$ such that for all $\ell \in \N$ and $i_1, \ldots, i_\ell \in I$, we have
		$|\lambda(x_{i_1} \cdots x_{i_\ell})| \leq R_{i_1} \dots R_{i_\ell}$. We equip $\Sigma_{I,\mathbf{R}}$ with the weak-$\ast$ topology, viewed as a subset of the dual of the vector space $\C \scal{x_i : i \in I}$.  Let $\Sigma_I = \bigcup_{\mathbf{R}} \Sigma_{I,\mathbf{R}}$ and equip it also with the weak-$*$ topology.
		
		For a self-adjoint $I$-tuple $\mathbf{X}$ from $\cM$, define the \textit{law} of $X$ as the map $\mu_{\mathbf{X}} : \C \scal{x_i: i \in I} \rightarrow \C$ given by $p(x) \mapsto \tau(p(X))$.
	\end{definition}
	
	\begin{definition}
		Given $\bX = (X_i)_{i \in I}$ a tuple of self-adjoint elements of $(\cM, \tau)$, we say that a sequence of $I$-tuples $\bX^{(n)}$ \textit{converges to $\bX$ in non-commutative law} if $\mu_{\bX^{(n)}} \rightarrow \mu_\bX$. Note that the elements of the sequence $\bX^{(n)}$ are from some tracial $\mathrm{W}^*$-algebras $(\cM^{(n)}, \tau_n)$, but these may all be different from $(\cM, \tau)$.
	\end{definition}
	
	We now describe an analog of independence in the non-commutative setting: \emph{free independence}, which is closely related to free products of von Neumann algebras.  We present some definitions and facts relevant to free independence; for further background, see for instance \cite{VDN1992}, \cite[\S 5]{AGZ2009}, \cite{MingoSpeicher}.
	
	\begin{definition}
		Let $\cM$ be tracial von Neumann algebra, and let $(\cM_i)_{i \in I}$ be a family of von Neumann subalgebras.  We say that $(\cM_i)_{i \in I}$ are \emph{freely independent} if for every $i_1$, \dots, $i_\ell \in I$ with $i_j \neq i_{j+1}$, whenever $a_j \in \cM_{i_j}$, then
		\[
		\tau \left[ (a_1 - \tau(a_1)) \dots (a_\ell - \tau(a_\ell)) \right] = 0.
		\]
		Moreover, if $\mathbf{X}_i$ for $i \in I$ are tuples (each with their own index set) from $\cM$, we say that $\mathbf{X}_i$ are freely independent if the von Neumann subalgebras $\mathrm{W}^*(\mathbf{X}_i)$ are freely independent.
	\end{definition}
	
	Given any $(\cM_i)_{i \in I}$, there exists a \emph{free product} $\cM = *_{i \in I} \cM_i$, that is, a tracial von Neumann algebra $\cM$ containing $\cM_i$'s as subalgebras that are freely independent of each other; see e.g.\ \cite[p. 351-352]{AGZ2009}.  Moreover, if $\mathbf{X}$ and $\mathbf{Y}$ are freely independent self-adjoint tuples, then the non-commutative law of the joint tuple $(\mathbf{X},\mathbf{Y})$ is uniquely determined by those of $\mathbf{X}$ and $\mathbf{Y}$; see e.g.\ \cite[\S 1.12, Proposition 13]{MingoSpeicher}.

	The analog in free probability of a tuple of Gaussian random variables is a \emph{standard free semicircular family}.  We say that $\mathbf{S} = (S_1,\dots,S_m)$ is a \emph{standard free semicircular family} if $S_1$, \dots, $S_m$ are freely independent from each other and each $S_i$ has the standard Wigner semicircular distribution, i.e.,
	\[
	\tau(p(S_i)) = \frac{1}{2\pi} \int_{-2}^2 p(x) \sqrt{4 - x^2}\,dx
	\]
	for every polynomial $p$.
	
	Voiculescu's asymptotic freeness theory shows how free independence arises in the large-$n$ limit from independence of certain random $n \times n$ matrices.  Correspondingly, the free semicircular family arises from certain Gaussian family of random matrices.  Here we use the following notation and terminology.  
	
	\begin{notation} \label{not: matrix space}
		\emph{} 
		\begin{itemize}
			\item Note that $M_n(\C)_{\sa}^m$ can be equipped with a real inner product $\ip{\mathbf{X},\mathbf{Y}}_{\tr_n} = \sum_{j=1}^m \tr_n(X_j^*Y_j)$, where $\tr_n$ is the normalized trace on $M_n(\C)$.
			\item Being a real inner product space of dimension $mn^2$, there is a linear isometry $M_n(\C)_{\sa}^m \to \R^{mn^2}$.  The \emph{Lebesgue measure} on $M_n(\C)_{\sa}^m$ is the measure obtained by transferring the Lebesgue measure on $\R^{n^2}$ by such an isometry (this is independent of the choice of isometry by invariance of Lebesgue measure).  We denote the Lebesgue measure by $\nu_n$.
			\item A \emph{self-adjoint random matrix} means a random variable on some probability space, taking values in $M_n(\C)_{\sa}$.
			\item Then a \emph{GUE matrix} is a random self-adjoint matrix with probability density $(1/Z_n) e^{-n^2 \norm{x}_2^2/2}$ with respect to Lebesgue measure, where $Z_n$ is a normalizing constant.  The GUE matrix is self-adjoint, with diagonal entries real normal random variables of mean zero and variance $1/n$, and the off-diagonal entries have real and imaginary parts that are normal random variables of mean zero and variance $1/2n$.   See also Section \ref{subsec: normalization}.
		\end{itemize}        
	\end{notation}
	
	The main results that we will need about asymptotic freeness can be summarized as follows.
	
	\begin{theorem} \label{thm: asymptotic free independence}
		Let $\cM$ be a tracial von Neumann algebra and $\mathbf{Y}$ a self-adjoint $\N$-tuple from $\cM$.  Let $S_1$, \dots, $S_m$ be a standard free semicircular family freely independent from $\mathbf{Y}$.
		
		Let $Y_1^{(n)}$, $Y_2^{(n)}$, \dots be random self-adjoint matrices such that
		\begin{enumerate}[(1)]
			\item Almost surely $\limsup_{n \to \infty} \norm{Y_j^{(n)}}_{\operatorname{op}} < \infty$.
			\item Almost surely $\mathbf{Y}^{(n)}$ converges to $\mathbf{Y}$ in non-commutative law.
		\end{enumerate}
		Let $S_1^{(n)}$, \dots, $S_m^{(n)}$ be independent $n \times n$ GUE matrices.  Then
		\begin{enumerate}[(a)]
			\item $\lim_{n \to \infty} \E \norm{S_j^{(n)}}_{\operatorname{op}}^j = 2^j$.
			\item Almost surely $(\mathbf{S}^{(n)},\mathbf{Y}^{(n)})$ converges to $(\mathbf{S},\mathbf{Y})$ in non-commutative law.
		\end{enumerate}
	\end{theorem}
	
	This theorem combines several known statements in random matrix theory.  Theorem \ref{thm: asymptotic free independence} (b) is \cite[Theorem 2.2]{Voiculescu1991}.
	
	For claim (a), since the behavior of the largest eigenvalue of a GUE matrix has been studied in depth, there are several sources we could deduce this from.  For instance, \cite[p. 24]{AGZ2009} shows the following:  Let $k(n)$ be an integer such that $k(n)^{c_1} / n \to 0$ and $k(n) / \log n \to \infty$ as $n \to \infty$, where $c_1$ is a certain positive constant.  Then we have for sufficiently large $n$ that
	\[
	\mathbb{P}(\Tr_n((S^{(n)})^{2k(n)}) > (2 + \delta)^{2k(n)} ) \leq 2n 4^{k(n)}.
	\]
	Then by applying H\"older's inequality on the underlying probability space
	\[
	\mathbb{E}[ \norm{S^{(n)}}_{\operatorname{op}}^j ] \leq \mathbb{E}[ \norm{S^{(n)}}_{\operatorname{op}}^{2k(n)} ]^{j/2k(n)} \leq \mathbb{E} \left[ \Tr_n((S^{(n)})^{2k(n)} ) \right]^{j/2k(n)} \leq \left( 2n 4^{k(n)} \right)^{j/2k(n)} = (2n)^{j/2k(n)} 2^j.
	\]
	Taking the limit as $n \to \infty$, we obtain Theorem \ref{thm: asymptotic free independence} (a).
	
	\subsection{Ultrafilters, ultralimits, and ultraproducts} \label{sec: ultra prelims}
	
	As mentioned in the introduction, we define conditional microstates free entropy using ultrafilters to highlight the connection with embeddings into matrix ultraproducts.  Here we quickly review the notions of ultrafilters, ultralimits, and the tracial ultraproduct construction for von Neumann algebras.  For further reference, see \cite{Goldbring2022ultrafilters}, \cite[Appendix A]{BrownOzawa2008}, \cite[\S 5.3]{GJNS2022}, \cite[\S 5.7]{Ioana2023}.
	
	A \emph{filter} $\mathcal{U}$ on an index set $I$ is a collection of subsets of $I$ such that 
	\begin{itemize}
		\item $\emptyset \not \in \cU$, $I \in \cU$;
		\item whenever $A, B \in \cU$, then $A \cap B \in \cU$;
		\item whenever $A \subseteq B$ and $A \in \cU$, then $B \in \cU$;
		\item If in addition $\cU$ is maximal, i.e. for any $A \subseteq I$, we have either $A \in \cU$ or $I \setminus A \in \cU$, then we say $\cU$ is an \emph{ultrafilter}.
	\end{itemize}
	We say that $\cU$ is principal if $\cU = \{A \subseteq I: i \in A\}$ for some $i \in I$. Otherwise we say $\cU$ is \emph{non-principal}.

	Later on in some proofs, we will need the following easy consequence for ultrafilters on $\N$. This fact is well-known, but we provide a self-contained proof for the reader's convenience.
	
	\begin{fact} \label{fact: IIP ultrafilters exist}
		Suppose $\{A_n\}_{n \in \N}$ is a nested decreasing sequence of non-empty subsets of $\N$ such that $\bigcap_{n \in \N} A_n = \emptyset$. Then there exists a non-principal ultrafilter $\cU$ such that $A_n \in \cU$ for all $n$.
	\end{fact}
	
	\begin{proof}
		Consider $\mathcal{F} = \{A_n\}_{n \in \N} \, \cup \, \{B \subseteq \N : B \supseteq A_n \text{ for some $n$} \}$. Then $\mathcal{F}$ is a filter by construction, so one can find an ultrafilter $\cU$ extending $\mathcal{F}$ by Zorn's lemma. Note that the ultrafilter $\cU$ cannot be principal, because $A_n \in \cU$ for all $n$, but $\bigcap_{n \in \N} A_n = \emptyset \not \in \cU$.
	\end{proof}
	
	In this paper, all ultrafilters will be on $\N$. Given a sequence $(x_n)_{n \in \N} \subseteq \R$, we say that \emph{$x_n$ converges to $x$ along the ultrafilter $\cU$}, denoted $\lim_{n \to \cU} x_n = x$, if for every $\epsilon > 0$, we have 
	$\{k \in \N : |x_k - x| < \epsilon\} \in \cU$. 
	This generalizes the notions of the usual limit, $\limsup$, and $\liminf$.
	
	Finally, we discuss the tracial ultraproduct construction. Given a sequence of tracial von Neumann algebras $(\cM_n, \tau_n)$, we consider the set of uniformly bounded sequences ${\{(x_n) \in \prod_\N \cM_n : \sup_n \norm{x_n} < \infty\}}$, denoted by $\ell^\infty((\cM_n)_{n \in \N})$. 
	We then take the separation-completion of $\ell^\infty((\cM_n)_{n \in \N})$ with respect to the norm $\norm{x}_2 = \tau(x^* x)^{1/2}$, where $\tau((x_n)_{n \in \N}) = \lim_{n \to \cU} \tau_n(x_n)$. The resulting object is denoted by $(\prod_\cU \cM_n, \tau)$ and is again a tracial von Neumann algebra, whose elements we write as $(x_n)_\cU$ for a representing sequence $(x_n)_{n \in \N}$. For further details on this construction, see \cite[\S 14.4]{Goldbring2022ultrafilters}.
	
	In this paper, we only focus on matrix ultraproducts $\mathbb{M}_{\cU} := \prod_\cU M_n(\C)$. Our motivation for phrasing the main results in terms of ultraproducts stems from the fact that they make approximate embeddings exact:
	\begin{fact}[{See also \cite[Lemma 5.10]{GJNS2022}}] \label{fact/def: matrix ultraproduct}
		Fix a tracial von Neumann algebra $(\cM, \tau)$ and a non-principal ultrafilter $\cU$.
		Let $\bX = (X_1, X_2,\dots)$ be a tuple from $\cM_{\sa}$. 
		Let $(\prod_\cU M_n(\C), \tr)$ be a matrix ultraproduct with trace $\tr = \lim_{n \to \cU} \tr_n$.
		If $(\bX^{(n)})_{n \in \N} = ((X_1^{(n)}, X_2^{(n)},\dots)_{n \in \N}$ is a sequence of $m$-tuples of self-adjoint matrices in $M_n(\C)_{\sa}$ such that $\sup_n \norm{X_j}_{\operatorname{op}} < \infty$ for each $j$ and $\bX^{(n)}$ converges in non-commutative law to $\bX$, then there is a unique trace-preserving embedding ${\pi: (\mathrm{W}^*(\bX), \tau) \rightarrow (\prod_\cU M_n(\C), \tr)}$ which sends $\bX$ to $(\bX^{(n)})_\cU$.
	\end{fact}
	This is easy to check: indeed, since traces are normal linear functionals, it is enough to consider monomials $x_{i_1} \cdots x_{i_k}$. But then simply note that by convergence of joint moments we have that 
	\[
	\tau(X_{i_1} \cdots X_{i_k}) = \lim_{n \to \cU} \tr_n(X_{i_1}^{(n)} \cdots X_{i_k}^{(n)}) = \tr((X_{i_1}^{(n)} \cdots X_{i_k}^{(n)})_\cU) = 
	\tr(\pi(X_{i_1} \cdots X_{i_k})).
	\]
	
	\section{Background on classical entropy} \label{sec: classical entropy}
	
	\subsection{Classical Entropy and Fisher Information}
	
	We briefly review only basic definitions of classical entropy and Fisher information. Many of these facts are from \cite{Barron1996, SW1949, Stam1959}, and an exposition of entropy suited to the random matrix context is given in \cite[Chapter 7 \& 8]{MingoSpeicher}.
	% \textcolor{gray}{Get more classical references to put here.}
	
	\begin{definition}
		Let $\mu$ be a measure on $\R^m$ with absolutely continuous density $\rho$ with respect to Lebesgue measure, i.e. $d\mu = \rho \, dx$. The \textit{entropy} of $\mu$ is defined to be
		\[
		h(\mu) := \int_{\R^m} - \rho \log \rho \, dx,
		\]
		whenever the integral is defined. If $\mu$ does not have a density $\rho$, then we set $h(\mu) := - \infty$. 
		If $X$ is a random variable with distribution $\mu$, then we also write $h(X)$ in place of $h(\mu)$.
	\end{definition}
	
	% In the classical case, the entropy is the rate function for a large deviation principle. More specifically, it measures the exponential unlikelihood of deviation from $\mu$ in the asymptotic limit of the distribution of $\frac{1}{n} \sum_{j=1}^n \delta_{X_j}$, where $X_j$ are independent and identically distributed with common distribution $\mu$. 
	
	We will need some facts about classical entropy in some of the proofs to follow.
	We list these facts here for the readers convenience; they come from \cite[Lemma B.5(i) \& (ii)]{ST2022}.
	\begin{fact}[Entropy controlled by partition] \label{fact: entropy controlled by ptn}
		Let $X$ be a random variable in $\R^m$ with law $\mu$, and let $(S_j)_{j=1}^\infty$ be a measurable partition of $\R^m$. Then
		\[
		h(\mu) \leq \sum_{j=0}^\infty \mu(S_j) \log \operatorname{Leb}(S_j) - \sum_{j=0}^\infty \mu(S_j) \log \mu(S_j),
		\]
		where $\operatorname{Leb}$ denotes the Lebesgue measure.
	\end{fact}
	
	\begin{fact}[Entropy controlled by variance] \label{fact: entropy controlled by var}
		If $X$ is a random variable taking values in $\R^m$ with finite variance $\operatorname{Var}(X)$, then
		\[
		h(X) \leq \frac{d}{2} \log(2\pi e \operatorname{Var}(X)/d).
		\]
	\end{fact}
	
	We recall the classical Fisher information of $\mu$, denoted $\mathcal{I}(\mu)$, which is defined below as the derivative at time zero of a Brownian motion starting at $\mu$. 
	
	Let $\gamma_t$ be the multivariate Gaussian measure on $\R^m$ with covariance matrix $tI$. Suppose $\mu$ is a measure with smooth density $\rho > 0$ on $\R^m$, and set $\mu_t = \mu \ast \gamma_t$ with associated density $\rho_t$. We may compute via integration by parts that $\partial_t h(\mu_t) = \int \norm{\nabla \rho_t / \rho_t}^2 \ d\mu_t$. Evaluating this at $t = 0$, we obtain that
	\[
	\mathcal{I}(\mu) = \int \norm{\nabla \rho / \rho}^2 \ d\mu = \E \left[ \norm*{- \nabla \rho / \rho }^2 \right].
	\]
	Rewriting the integral above as the expected value on the right hand side helps with moving to the non-commutative case. 
	Namely, $\Xi := - \frac{\nabla \rho}{ \rho} (\bX)$ satisfies the relation 
	\begin{equation} \label{cl-score-fxn-eqn}
		\E[ \scal{\Xi, f(\bX)}_{\C^m}] = \E[ \nabla \cdot  f(\bX)] \text{ for all } f \in C_c^\infty (\R^m;\C^m),
	\end{equation}
	%or equivalently the integration-by-parts relation $\E[\Xi \cdot f(\bX)] = \E[ \nabla f(\bX)]$ for $f \in C_c^\infty(\R^m;\C)$
	see \cite[\S 12]{JekelThesis} for details. This latter notion makes sense even when $\mu$ does not have a smooth density, and motivates the following definition.
	
	\begin{definition} \label{def: classical score function}
		Suppose $\bX$ is an $\R^m$-valued random variable on the probability space $(\Omega, \mathbb{P})$. If there is a random variable $\Xi \in L^2(\Omega, \mathbb{P})$ satisfying \ref{cl-score-fxn-eqn} and each $\Xi_j$ is in the closure of $\{f(\bX) \ : \ f \in C_c^\infty (\R^m) \}$ in $L^2(\Omega, \mathbb{P})$, then we say $\Xi$ is the \textit{score function} for $\bX$.
	\end{definition}
	
	\begin{definition}
		The \textit{Fisher information} $\mathcal{I}(\mu)$ is defined as $\E[ |\Xi|^2]$ if $\bX \sim \mu$ and $\Xi$ is a score function for $\bX$. If no score function for $\bX$ exists, then we set $\mathcal{I}(\mu) := \infty$.
	\end{definition}
	
	Note that the definition of Fisher information aligns with the integral $\int \norm{\nabla \rho / \rho}^2 \ d\mu$ if $\mu$ has smooth density $\rho$, but the given definition is more general and will be directly analogous to the free Fisher information given in Definition $\ref{def: cond free Fisher info}$.
	
	We remark that the integration-by-parts relation can be extended to more general functions than $C_c^\infty$.  Specifically, if $\mu$ has finite moments, then we can use test functions of polynomial growth.  This will become important later when we relate classical and free Fisher information (see Corollary \ref{cor: IBP matrix version} and Proposition \ref{prop: cond Phi* upper bound}).
	
	\begin{lemma} \label{lem: poly growth IBP}
		Let $\bX$ be an $\R^m$-valued random variable with finite moments, and let $\Xi$ be a score function for $\bX$.  Let $f: \R^m \to \C^m$ be a smooth function such that $|f(x)| \leq A(1 + |x|^2)^k$ and $\norm{Df(x)} \leq B(1 + |x|^2)^k$ for some $A, B > 0$ and $k \in \N$.  Then $\E[ \scal{\Xi, f(\bX)}_{\C^m}] = \E[ \nabla \cdot f(\bX)]$.
	\end{lemma}
	
	\begin{proof}
		Let $\phi: \R^m \to [0,1]$ be a $C_c^\infty$ function such that $\phi(0) = 1$, and for $t \in (0,1]$, let
		\[
		f_t(x) = f(x) \phi(tx).
		\]
		Note that $|f_t(x)| \leq A(1 + |x|^2)^k$ and $\lim_{t \to 0} f_t(x) = f(x)$.  Also,
		\[
		Df_t(x) = Df(x) \phi(tx) + t f(x) (\nabla \phi(x))^*,
		\]
		where $(\nabla \phi(x))^*$ is the row vector obtained by transposing $\phi(x)$,
		and hence
		\[
		\norm{D f_t(x)} \leq \norm{Df(x)} + t |f(x)| \norm{\nabla \phi}_{L^\infty} \leq (B +  A\norm{\nabla \phi}_{L^\infty}) (1 + |x|^2)^k,
		\]
		and $\lim_{t \to 0} Df_t(x) = Df(x)$.  Since $f_t \in C_c^\infty(\R^m;\C^m)$, we have
		\[
		\E [\scal{\Xi,f_t(\mathbf{X})}_{\C^m} ]= \E [\nabla \cdot f_t(\mathbf{X})].
		\]
		Now we will take $t \to 0$ and apply the dominated convergence theorem (on the underlying probability space).  Note that
		\[
		|\ip{\Xi,f_t(\mathbf{X})}_{\C^m}| \leq A |\Xi| (1 + |\mathbf{X}|^2)^k,
		\]
		and the function on the right is in $L^1$ because $\Xi$ and $(1 + |\mathbf{X}|^2)^k$ are in $L^2$; indeed, $\E (1 + |\mathbf{X}^2|)^{2k} < \infty$ because we assumed $\mathbf{X}$ has finite moments.  Hence, by the dominated convergence theorem,
		\[
		\lim_{t \to 0} \E [\scal{\Xi,f_t(\mathbf{X})}_{\C^m} ]= \E [\scal{\Xi,f(\mathbf{X})}_{\C^m}].
		\]
		Similarly,
		\[
		|\nabla \cdot f_t(\mathbf{X})| \leq m \norm{Df_t(\mathbf{X})} \leq m(B +  A\norm{\nabla \phi}_{L^\infty}) (1 + |\mathbf{X}|^2)^k,
		\]
		and the latter has finite expectation because $\mathbf{X}$ has finite moments.  Therefore, by dominated convergence,
		\[
		\E [\nabla \cdot f(\mathbf{X})] = \lim_{t \to 0} \E [\nabla \cdot f_t(\mathbf{X})] = \lim_{t \to 0} \E [\scal{\Xi, f_t(\mathbf{X})}_{\C^m} ]= \E [\scal{\Xi, f(\mathbf{X})}_{\C^m} ]. \qedhere
		\]
	\end{proof}
	
	Finally, the classical entropy can be recovered from an appropriate integral of $\mathcal{I}(\mu_t)$, where $\mu_t$ is the convolution of $\mu$ with a Gaussian measure $\gamma_t$ with covariance matrix $tI$.  This formula motivated Voiculescu's definition of $\chi^*$ (we give a conditional version of this in Definition \ref{def: relative nm free ent}), and it will be important for the proof of the main result.  The fact is standard and a proof can be found for instance in \cite[\text{Lemma } 12.1.4]{JekelThesis}.
	
	\begin{lemma} \label{lem: ent from I()}
		Let $\mu$ be a probability measure on $\R^m$ with finite variance and density $\rho$, and let $\gamma_t$ be the centered Gaussian measure with covariance matrix $tI$. Then,
		\[
		h(\mu * \gamma_t) - h(\mu) = \frac{1}{2} \int_0^t \mathcal{I}(\mu * \gamma_s) ds,
		\]
		and
		\[
		h(\mu) = \frac{1}{2} \int_0^\infty \left( \frac{m}{1 + t} - \mathcal{I}(\mu \ast \gamma_t) \right) \ dt + \frac{m}{2} \log(2\pi e).
		\]
	\end{lemma}

	Fisher information has the following property with respect to sums of independent random variables.  This fact is standard and quick to prove (see for instance \cite[Lemma 12.1.3]{JekelThesis}).
	
	\begin{lemma} \label{lem: Fisher info convolution}
		Let $\mu$ and $\nu$ be measures in $\R^m$.  Then $\mathcal{I}(\mu * \nu) \leq \min(\mathcal{I}(\mu),\mathcal{I}(\nu))$.
	\end{lemma}
	
	\begin{corollary} \label{cor: Fisher info decreasing}
		Let $\mu$ be a probability measure on $\R^m$ and let $\gamma_t$ be the standard Gaussian measure with covariance matrix $tI$.  Then $t \mapsto \mathcal{I}(\mu * \gamma_t)$ is decreasing.
	\end{corollary}
	
	This follows immediately from the previous lemma since if $s < t$, then $\mathcal{I}(\mu * \gamma_t) = \mathcal{I}(\mu * \gamma_s * \gamma_{t-s}) \leq \mathcal{I}(\mu * \gamma_s)$.
	
	\subsection{Classical Conditional Entropy and Fisher Information} \label{sec: cond fisher info}
	
	\begin{definition}
		Suppose that $\bX = (X_1,\dots,X_m)$ and $\bY = (Y_1,Y_2,\dots)$ are random variables on some probability space $\Omega$.  If the joint distribution of $(\bX, \bY)$ has a disintegration as \\ $\rho(x \mid y) \,dx_1\,\dots\,dx_m \,d\mu(y)$ for some measure $\mu$ on $\R^{\N}$, then we define the \textit{conditional entropy}
		\[
		h(\bX \mid \bY) = - \int_{\Omega_Y} \int_{\Omega_X} \rho(x\mid y) \log \rho(x\mid y) \, dx_1\, \dots \,dx_m \, d\mu(y).
		\]
		If no such disintegration exists, we set $h(\mathbf{X} \mid \mathbf{Y}) = -\infty$.
	\end{definition}

	\begin{definition} \label{def: conditional fisher info} 
		If there is a random vector $\Xi \in L^2(\Omega)^m$ satisfying the integration-by-parts relation:
		for any $n \in \N$, any smooth compactly supported function $f: \R^m \times \R^n \rightarrow \R^m$, and any indices $i_1, \ldots, i_n$,
		\[
		\E \scal{\Xi, f(\bX,Y_{i_1}, \ldots Y_{i_n})} = \E[ \operatorname{div}_X f(\bX,Y_{i_1}, \ldots, Y_{i_n})],
		\]
		then $\Xi$ is unique in $L^2(\Omega)^m$ and is defined by evaluating $- \nabla_\bX \log(\rho_{\bX \mid 
			\bY})$ on the random variable $(\bX, \bY)$. In this case, we call $\Xi$ the \textit{score function for $\bX$ given $\bY$}.
		We also define the \textit{conditional Fisher information} by 
		\[
		\mathcal{I}(X \mid Y) = \E |\Xi|^2, 
		%       = \int_{\Omega_X \times \Omega_Y} \rho_{X,Y}(x,y) \cdot \left| \frac{\nabla_X \rho_{X\mid Y} (x\mid y)}{\rho_{X\mid Y}(x\mid y)} \right|^2 \ d\mu_X \ d\nu_Y,
		\]
		whenever the above quantity exists, and $\infty$ otherwise.
	\end{definition}
	
	When $(\bX, \bY)$ is a random variable in $\R^m \times \R^n$, the conditional Fisher information describes the rate of change of $h(\bX + t^{1/2} \bS \mid  \bY)$, where $\bS$ is a Gaussian random variable in $\R^m$ with covariance matrix $I_m$, independent from $(\bX,\bY)$. More concretely, one can show that  $\mathcal{I}(\bX + t^{1/2} \bS \mid  \bY)$ is well-defined and finite for $t > 0$ and
	\[
	\frac{d}{dt} h \left( \bX + t^{1/2} \bS \mid \bY \right) = \frac{1}{2} \mathcal{I}\left( \bX + t^{1/2} \bS \mid \bY \right).
	\]
	
	\subsection{Normalization for random matrix theory} \label{subsec: normalization}
	
	In random matrix theory, we consider probability measures on $M_n(\C)_{\sa}^m$, the space of $m$-tuples of self-adjoint matrices.  Per Notation \ref{not: matrix space}, $M_n(\C)_{\sa}^m$ is viewed as a real inner product space of dimension $mn^2$ with the inner product $\ip{\cdot,\cdot}_{\tr_n}$.  Thus, concepts such as Lebesgue measure, gradient, divergence, entropy, and Fisher's information are defined for random matrices by coordinatizing $M_n(\C)_{\sa}^m$ using an orthonormal basis with respect to $\ip{\cdot,\cdot}_{\tr_n}$.  We remark that coordinatizing using $\ip{\cdot,\cdot}_{\Tr_n}$ versus $\ip{\cdot,\cdot}_{\tr_n}$ yields \emph{different} normalizations for some of these quantities.  Indeed, an orthonormal basis for $M_n(\C)_{\sa}$ with respect to $\ip{\cdot,\cdot}_{\tr_n}$ would be
	\[
	\mathcal{E} = \{\sqrt{n} E_{j,j}\}_{j=1}^n \cup \{\sqrt{n/2}(E_{j,k} + E_{k,j}) \}_{1 \leq j < k \leq n} \cup \{i \sqrt{n/2}(E_{j,k} - E_{k,j}) \}_{1 \leq j < k \leq n},
	\]
	while an orthonormal basis with respect to $\ip{\cdot,\cdot}_{\Tr_n}$ would remove the $\sqrt{n}$ factors.  Thus, the convention using $\Tr_n$ is most convenient for entrywise computations in random matrix theory, but we will follow the convention using $\tr_n$ because the normalized trace is what relates to the free probabilistic limit as $n \to \infty$, and most of our computations are coordinate-free, i.e.\ we work with inner products and linear transformations and do not need to refer to the matrix entries.
	
	In any case, a dimensional renormalization of the entropy, score function, and Fisher information are needed in order to discuss the large-$n$ limit.  One can see this for instance from computing the entropy of a GUE; see \cite[Appendix B]{ST2022}.
	
	\begin{definition}
		Let $(\bX, \bY)$ be a random variable in $M_n(\C)_{sa}^m \times M_n(\C)_{sa}^\N$, with density $\rho_{\bX,\bY}$. Then, define the \textit{normalized conditional entropy} by
		\[
		h^{(n)}(\bX \mid \bY) := \frac{1}{n^2} h(\bX \mid \bY) + m \log n,
		\]
		and the \textit{normalized conditional Fisher information} as $\mathcal{I}^{(n)}(\bX \mid \bY) := n^{-4} \mathcal{I}(\bX \mid \bY)$, where $h$ and $\mathcal{I}$ are understood with respect to the inner product $\ip{\cdot,\cdot}_{\tr_n}$ coming from the normalized trace.
	\end{definition}
	
	Further detail about this normalization can be found in \cite[\S 16]{JekelThesis} and \cite[Appendix B]{ST2022}, which both take the classical quantities to be defined based on $\ip{\cdot,\cdot}_{\tr_n}$.  If the classical quantities are instead defined based on $\ip{\cdot,\cdot}_{\Tr_n}$, the equations change slightly.  Denoting (for the moment) the two versions of classical entropy and Fisher information by subscripts of $\tr_n$ and $\Tr_n$ respectively, we have
	\begin{align*}
		h^{(n)} &= \frac{1}{n^2} h_{\tr_n} + m \log n = \frac{1}{n^2} h_{\Tr_n} + \frac{m}{2} \log n \\
		\mathcal{I}^{(n)} &= \frac{1}{n^4} \mathcal{I}_{\tr_n} = \frac{1}{n^3} \mathcal{I}_{\Tr_n}.
	\end{align*}
	For the $\Tr_n$ version, see for instance \cite[Section 6.2]{Jekel_ConditionalEntropy2020}.  An advantage of basing everything on $\ip{\cdot,\cdot}_{\tr_n}$ as we do here is that $n^2 = \dim_{\R} M_n(\C)_{\sa}$ appears in a natural way in the formula.
	
	Using the normalized versions, the relationship between entropy and Fisher information can be stated as follows:
	\[
	\partial_t h^{(n)}(\bX + t^{1/2} \bS \mid  \bY) = \frac{1}{2n^4} \mathcal{I}(\bX + t^{1/2} \bS \mid \bY) = \frac{1}{2} \mathcal{I}^{(n)}(\bX + t^{1/2} \bS \mid  \bY),
	\]
	assuming $\bX$ has finite variance and $t > 0$.  To prove this, we must relate the GUE tuple $\mathbf{S}$ with the standard Gaussian vector in the inner product space $M_n(\C)_{\sa}^m$ with $\ip{\cdot,\cdot}_{\tr_n}$.  A standard Gaussian vector $\mathbf{G}$ can be constructed as
	\[
	\mathbf{G} = \sum_{E \in \mathcal{E}} Z_E E,
	\]
	where $\mathcal{E}$ is an orthonormal basis for $M_n(\C)_{\sa}^m$ with $\ip{\cdot,\cdot}_{\tr_n}$, and $Z_E$ are independent standard normal random variables.\footnote{Note that the definition of standard Gaussian vector depends on the choice of inner product. 
		When working with $\ip{\cdot,\cdot}_{\tr_n}$ rather than $\ip{\cdot,\cdot}_{\Tr_n}$, the standard Gaussian vector $\mathbf{G}$ will have matrix entries with variance $n$.  That is, the standard Gaussian vector in $M_n(\C)_{\sa}$ with $\ip{\cdot,\cdot}_{\tr_n}$ is $\sqrt{n}$ times the standard Gaussian vector in $M_n(\C)_{\sa}^m$ with $\ip{\cdot,\cdot}_{\Tr_n}$ which would have entries of variance $1$.}  One thus has that the total variance $\mathbb{E} \norm{\mathbf{G}}_2^2 = \dim_{\R} M_n(\C)_{\sa}^m = mn^2$.  A GUE $m$-tuple is thus obtained as $\bS = (1/n) \mathbf{G}$.  Hence,
	\[
	\partial_t h^{(n)}(\bX + t^{1/2} \bS \mid  \bY) = \frac{1}{n^2} \partial_t h(\bX + (t/n^2)^{1/2} \mathbf{G} \mid \bY ) = \frac{1}{2n^4} \mathcal{I}(\mathbf{X} + (t/n^2)^{1/2} \mathbf{G} \mid \bY).
	\]
	As a consequence,
	\begin{equation}\label{eqn: increments of h is integral of I()}
		h^{(n)}(\bX + t^{1/2} \bS \mid  \bY) - h^{(n)}(\bX \mid \bY) = \frac{1}{2} \int_0^t \mathcal{I}^{(n)} (\bX + u^{1/2} \bS \mid  \bY) \, du,
	\end{equation}
	and we may also recover the conditional entropy from the conditional Fisher information via the following integral formula, exactly as in Lemma \ref{lem: ent from I()}:
	\begin{equation} \label{eqn: recover conditional entropy from conditional I()}
		h^{(n)}(\bX \mid \bY) = \frac{1}{2} \int_0^\infty \left( \frac{m}{1+t} - \mathcal{I}^{(n)}(\bX + t^{1/2} \bS \mid \bY) \right) \, dt + \frac{m}{2} \log 2 \pi e.
	\end{equation}
	
	\section{Conditional Microstates Free Entropy} \label{sec: conditional microstates}
	
	\subsection{Definition and properties}
	
	Conditional microstates entropy $\chi(X_1, \dots, X_n \mid Y_1,\dots,Y_m)$ was first defined by Voiculescu \cite{VoiculescuFE2}.  Later Shlyakhtenko \cite{Shlyakhtenko2002} gave a different definition using the supremum of measures of relative microstate spaces over the microstates for $Y$, rather than the average as Voiculescu had done.  We will define the conditional entropy with respect to a fixed embedding into a matrix ultraproduct, and we will show that taking the supremum over the embedding and the ultrafilter gives Shlyakhtenko's conditional microstates entropy (Lemma \ref{lem: Shl-sup}).  Moreover, in Section \ref{subsec: conditional microstates} we will describe the relationship between conditional microstates entropy and conditional classical entropy analogously to the random matrix interpretation of microstates free entropy given in \cite[Proposition B.7]{ST2022}.
	
	\begin{definition}[Conditional microstates free entropy via fixed relative microstates] \label{def: cond ms free ent}
		Let $(\cM, \tau)$ be a tracial von Neumann algebra, with $\mathcal{B} \subseteq \cM$ a separable von Neumann subalgebra.
		Let ${\bX = (X_1, \ldots, X_m)}$ be an $m$-tuple of self-adjoint elements in $(\cM, \tau)$, and fix a tuple of generators $\bY = (Y_j)_{j \in \N}$ for $\mathcal{B}$.  Let $\mu := \text{law}(\bX, \bY)$, and fix a sequence $\bY^{(n)} \in M_n(\C)^\N_{\sa}$ that converges in non-commutative law to $\bY$ with $\sup_n \norm{Y_j^{(n)}}_{\op} < \infty$ for each $j$.
		Then for any neighborhood $\mathcal{O}$ of $\mu$, and any tuple $\mathbf{R} = (R_1, R_2, \ldots, R_m) \in (0, \infty)^m$, we define the \textit{conditional microstate spaces}:
		\begin{multline*}
			\Gamma_{\mathbf{R}}( \mathcal{O}  \mid  \bY^{(n)} \rightsquigarrow \bY ) := \Big\{ \bX^{(n)} \in M_n(\C)_{\sa}^m \ \big| \ \norm{X_j^{(n)}}_{\op} \leq R_j \text{ for all } 1 \leq j \leq m, 
			\text{ and } \text{law}(\bX^{(n)}, \bY^{(n)}) \in \mathcal{O} \Big\},
		\end{multline*}
		where $\nu_n$ denotes the Lebesgue measure on $M_n(\C)_{\sa}^m$ given in Notation \ref{not: matrix space}.
		
		We define the \textit{conditional microstates free entropy of $\bX$ given $\bY^{(n)} \rightsquigarrow \bY$}, \\
		denoted $\chi(\bX  \mid  \bY^{(n)} \rightsquigarrow \bY)$ via:
		\[
		\chi_{\mathbf{R}}(\bX  \mid  \bY^{(n)} \rightsquigarrow \bY) := \inf_{\mathcal{O} \text{ nbhd of } \mu} \  \limsup_{n \rightarrow \infty} \left[ \frac{1}{n^2} \log \nu_n( \Gamma_{\mathbf{R}}^{(n)}( \mathcal{O}  \mid  \bY^{(n)} \rightsquigarrow \bY)) + m \log n \right].
		\]
		Finally, define 
		\[
		\chi(\bX  \mid  \bY^{(n)} \rightsquigarrow \bY) := \sup_{\mathbf{R} \in (0,\infty)^{m}} \chi_{\mathbf{R}}(\bX  \mid  \bY^{(n)} \rightsquigarrow \bY).
		\]
		Also, for any non-principal ultrafilter $\mathcal{U}$ of $\N$, we define $\chi^{\cU}(\bX  \mid  \bY^{(n)} \rightsquigarrow \bY)$ to be the same expression as above, with the $\limsup_{n \rightarrow \infty}$ replaced with $\lim_{n \rightarrow \cU}$.
	\end{definition}
	
	We now show that the value of $\chi_{\mathbf{R}}(\bX \mid \bY^{(n)} \rightsquigarrow \bY)$ does not depend on the parameter $\mathbf{R}$ as long as $\norm{X_j}_{\op} \leq R_j$.
	
	\begin{lemma} \label{lemma: Chi_R = Chi}
		Set $r_j = \norm{X_j}_{\op}$ and $\mathbf{r} = (r_1, \ldots, r_m)$.
		Then if $\mathbf{R}, \mathbf{R}' \in (0, \infty)^m$ are such that $r_j < R_j < R'_j$ for each $j = 1, \ldots, m$, then 
		\[
		\chi_{\mathbf{R}}(\bX  \mid  \bY^{(n)} \rightsquigarrow \bY) = \chi_{\mathbf{R}'}(\bX  \mid  \bY^{(n)} \rightsquigarrow \bY).
		\]
		Hence if $R_j > r_j$ for all $j$, then
		\[
		\chi_\mathbf{R}(\bX  \mid  \bY^{(n)} \rightsquigarrow \bY) = \chi(\bX  \mid  \bY^{(n)} \rightsquigarrow \bY).
		\]
		
		\begin{proof}
			This is the conditional analogue to \cite[Proposition 2.4]{VoiculescuFE2}; since the approximating sequence $\bY^{(n)}$ is fixed, the same argument showing inclusion of microstate spaces works: one can define an appropriate piecewise linear function $g$ such that for any $n \in \N$ and neighborhood $\mathcal{O}$ of $\mu$, there is some $n' \in \N$ and another neighborhood $\mathcal{O}'$ of $\mu$ such that
			\[
			G(\Gamma_{\mathbf{R}}^{(n')}( \mathcal{O}'  \mid  \bY^{(n)} \rightsquigarrow \bY )) \subseteq \Gamma_{\mathbf{R}'}^{(n)}( \mathcal{O}  \mid  \bY^{(n)} \rightsquigarrow \bY ),
			\]
			where $G(X_1, \ldots, X_n) = (g(X_1), \ldots, g(X_n))$ is applied componentwise in the matrices.
			This establishes that $\chi_{\mathbf{R}'}(\bX  \mid  \bY^{(n)} \rightsquigarrow \bY) \leq \chi_{\mathbf{R}}(\bX  \mid  \bY^{(n)} \rightsquigarrow \bY)$. The opposite inequality follows easily from the fact that $\mathbf{R} < \mathbf{R}'$ implies $\Gamma^{(n)}_{\mathbf{R}} ( \mathcal{O}  \mid  \bY^{(n)} \rightsquigarrow \bY)) \subseteq \Gamma^{(n)}_{\mathbf{R}'}( \mathcal{O}  \mid  \bY^{(n)} \rightsquigarrow \bY))$ for any fixed $\mathcal{O}$ and $n \in \N$.
		\end{proof}
	\end{lemma}
	
	In the rest of the paper, we write $\mathbf{R} > \norm{\bX}_{\op}$ as shorthand for $R_j > \norm{X_j}_{\op}$ for all $j = 1, \ldots m$.
	We now show that the value of the conditional microstates free entropy does not depend upon the choice of approximating microstates $\bY^{(n)}$, but rather only on the choice of induced embedding.
	
	\begin{definition}[Induced embedding] \label{def: induced emb}
		Let $\mathbb{M}_\cU := \prod_\cU M_n(\C)$ be a matrix ultraproduct (as in Fact \ref{fact/def: matrix ultraproduct}). Given $\bY$ a tuple of generators for a von Neumann subalgebra $(\mathcal{B}, \tau) \subseteq (\mathcal{M}, \tau)$, and a sequence of approximating matrix tuples $\bY^{(n)} \in M_n(\C)^\N_{\sa}$ converging to $\bY$ in non-commutative law, we define the \emph{induced embedding} $\iota_{\bY^{(n)} \rightsquigarrow \bY}$ as the unique trace-preserving $*$-homomorphism $\mathcal{B} \rightarrow \mathbb{M}_{\cU}$ such that $Y_j \mapsto [Y_j^{(n)}]_\cU$, which is well-defined by Fact \ref{fact/def: matrix ultraproduct}.
	\end{definition}
	
	We want to show, similarly to \cite[Theorem 2.15]{Shlyakhtenko2002}, that our conditional microstates free entropy does not depend upon the choice of generators for the subalgebra $\cB$. However, our definition of conditional microstates free entropy works with fixed microstates $\bY^{(n)}$ for the generators $\bY$, while the relative entropy in \cite{Shlyakhtenko2002} takes a supremum over possible values of $\bY^{(n)}$. Thus, we record that our conditional microstates free entropy $\chi^\cU (\bX \mid \bY^{(n)} \rightsquigarrow \bY)$ depends only on the induced embedding rather than the specific choice of $\bY^{(n)}$ and $\bY$.
	
	\begin{lemma}[Conditional microstates free entropy depends only on the subalgebra and the embedding]
		\label{lemma: cond microstates f.e. embedding}
		Let $\bY$ and $\bZ$ be countably infinite tuples of generators for the same subalgebra $(\mathcal{B}, \tau)$ of $(\mathcal{M}, \tau)$. 
		For any sequences of approximating microstates $\bY^{(n)}$ and $\bZ^{(n)}$ for $\bY$ and $\bZ$ respectively that give rise to the same induced embedding, i.e.
		$\iota := \iota_{\bY^{(n)} \rightsquigarrow \bY} = \iota_{\bZ^{(n)} \rightsquigarrow \bZ}$,
		we have
		\[
		\chi^\cU (\bX  \mid  \bY^{(n)} \rightsquigarrow \bY) = \chi^\cU(\bX  \mid  \bZ^{(n)} \rightsquigarrow \bZ).
		\]
		In other words, the conditional microstates free entropy depends only on the subalgebra $\mathcal{B}$ and the choice of embedding $\iota: \mathcal{B} \rightarrow \mathbb{M}_\cU$.
		
		\begin{proof}
			Since the hypotheses on $\mathbf{Y}$ and $\mathbf{Z}$ are symmetric, it suffices to establish the inequality \\
			${\chi^\cU(\bX \mid \bY^{(n)} \rightsquigarrow \bY) \leq \chi^\cU(\bX \mid \bZ^{(n)} \rightsquigarrow \bZ)}$.  From the definitions of these quantities, it is enough to show that for any neighborhood $\cO$ of law$(\bX, \bZ)$, and any $\mathbf{R} > \norm{\bX}_{\op}$ there is some neighborhood $\Tilde{\cO}$ of law$(\bX, \bY)$ satisfying
			\begin{equation} \label{eqn: set inclusion}
				\Gamma_{\mathbf{R}}^{(n)} ( \Tilde{\cO} \mid \bY^{(n)} \rightsquigarrow \bY) \subseteq \Gamma_{\mathbf{R}}^{(n)} ( \cO \mid \bZ^{(n)} \rightsquigarrow \bZ).
			\end{equation}
			Without loss of generality, we can take $\cO$ to be an element of the neighborhood basis for the weak-$*$ topology on the space of laws; in other words, assume there is some finite collection of non-commutative $\ast$-polynomials $\{q_1(\mathbf{x}, \mathbf{y}), \ldots, q_k(\mathbf{x}, \mathbf{y})\}$ and some tolerances $\varepsilon_j > 0$ so that
			\begin{align*}
				\Gamma_{\mathbf{R}}^{(n)} ( \cO \mid \bZ^{(n)} \rightsquigarrow \bZ) = \{\bX^{(n)} \in M_n(\C)_{\sa}^m \ : &\ | \tr_n(q_j(\bX^{(n)}, \bZ^{(n)})) - \tau(q_j(\bX, \bZ)) | < \varepsilon_j \text{ for all } 1 \leq j \leq k \\ 
				&\text{ and } \norm{\bX^{(n)}}_{\op} \leq \mathbf{R}\}.
			\end{align*}
			Set $\varepsilon = \min_{1 \leq j \leq k} \varepsilon_j$.
			
			By Fact \ref{fact: polynomial uniform continuity}, there exists some $\delta > 0$ and finite set $F$ of indices such that for every tracial von Neumann algebra $\cN$, for all $\bX' \in \cN_{\sa}^m$, $\bZ' \in \cN_{\sa}^{\N}$, $\bW' \in \cN_{\sa}^{\N}$ satisfying $\norm{X_i'}_{\op} \leq R$ and ${\max(\norm{\bZ_i'}_{\op},\norm{\bW_i'}_{\op}) \leq \norm{\bZ_i}_{\op}}$, whenever $\norm{\bZ_i' - \bW_i'}_2 < \delta$ for $i \in F$, then
			\[
			\norm{q_i(\bX',\bZ') - q_i(\bX',\bW')}_2 < \frac{\varepsilon}{3}.
			\]
			
			By Lemma \ref{lem: polynomial approximation}, for each $i$, there exists a polynomial $p_i$ such that $\norm{p_i(\bY) - \bZ_i}_2 < \epsilon$ and such that for every tracial von Neumann algebra $\cN$ and every $\bY'$ with $\norm{Y_i'}_{\op} \leq \norm{Y_i}_{\op}$, we have ${\norm{p_i(\bY_i')}_{\op} \leq \norm{Z_i}_{\op}}$.  Because the embeddings induced by $\bY^{(n)}$ and $\bZ^{(n)}$ are the same, we have that for $\cU$-many $n$, for all $i \in F$,
			\[
			\norm{p_i(\bY^{(n)}) - Z_i^{(n)}}_2 < \delta.
			\]
			This implies that for all $\bX^{(n)} \in M_n(\C)_{\sa}^m$ with $\norm{\bX^{(n)}}_{\op} \leq \mathbf{R}$, we have $\norm{p_i(\bY^{(n)})}_{\op} \leq \norm{Z_i}_{\op}$ and hence
			\[
			|\tr_n(q_i(\bX^{(n)},\bZ^{(n)})) - \tr_n(q_i(\bX^{(n)},p_i(\bY^{(n)})))| \leq \norm{q_i(\bX^{(n)},\bZ^{(n)}) - q_i(\bX^{(n)},p_i(\bY^{(n)}))}_2 < \frac{\varepsilon}{3}.
			\]
			Similarly,
			\[
			|\tau(q_i(\bX,\bZ)) - \tau(q_i(\bX,p_i(\bY)))| < \frac{\varepsilon}{3}.
			\]
			In particular, this means that for all $\bX^{(n)} \in M_n(\C)_{\sa}^m$ with $\norm{\bX^{(n)}}_{\op} \leq \mathbf{R}$, we have
			\begin{equation} \label{eqn: triangle for cond microstates f.e. embedding}
				\begin{split}
					|\tr_n(q_j(\bX^{(n)}, \bZ^{(n)})) - \tau(q_j(\bX, \bZ))| &\leq |\tau(q_j(\bX, p(\bY))) - \tau(q_j(\bX, \bZ))| \\
					&\quad + | \tr_n(q_j(\bX^{(n)}, \, \bZ^{(n)})) - \tr_n(q_j(\bX^{(n)}, \, p(\bY^{(n)})))| \\
					&\quad + |\tr_n(q_j(\bX^{(n)}, \, p(\bY^{(n)}))) - \tau(q_j(\bX, \, p(\bY)))| \\
					&\leq \frac{2 \varepsilon}{3} + |\tr_n(q_j(\bX^{(n)}, \, p(\bY^{(n)}))) - \tau(q_j(\bX, \, p(\bY)))|
				\end{split}
			\end{equation}

			%For the second term, recall that the embedding $\iota: \cB \rightarrow \mathbb{M}_\cU$ which maps
			%\begin{align*}
			%     p(\bY) \to [p(\bY^{(n)})]_\cU \quad \text{ and } \quad \bZ \to [\bZ^{(n)}]_\cU
			%\end{align*}
			%is trace-preserving, so in particular we have that since $|\tau(q_j(\mathbf{A}, p(\bY))) - \tau(q_j(\mathbf{A}, \bZ))| < \varepsilon/3$ uniformly for any input $\mathbf{A}$ in the $R$-ball of $\cM$, it follows that 
			%\[
			%  | \tr_{\mathbb{M}_\cU}(q_j(\bX, \, \iota(\bZ))) - \tr_{\mathbb{M}_\cU}(q_j(\bX, \, \iota(p(\bY))))| < \varepsilon/3,
			%\]
			%for any $\bX = [\bX^{(n)}]_\cU$ in $\mathbb{M}_\cU$ with $\norm{\bX^{(n)}}_{\op} < R$.
			%By the definition of the tracial ultraproduct construction of $\mathbb{M}_\cU$ we obtain
			%\[
			%   | \tr_n(q_j(\bX^{(n)}, \, \bZ^{(n)})) - \tr_n(q_j(\bX^{(n)}, \, p(\bY^{(n)})))| < \varepsilon/3
			%\]
			%for $\cU$-many indices $n$. \\
			
			Finally, to make this last term less than $\varepsilon/3$, we choose $\Tilde{\cO}$. Let $\Tilde{q_j}(\mathbf{x}, \mathbf{y}) = q_j(\mathbf{x}, p(\mathbf{y}))$. Then set $\Tilde{\cO}$ to be the neighborhood of law$(\bX, \bY)$ defined by approximating $\{\Tilde{q_1}, \ldots, \Tilde{q_k}\}$ up to tolerance $\varepsilon/3$. Then by construction, the last term appearing on the left-hand side of \ref{eqn: triangle for cond microstates f.e. embedding} is less than $\varepsilon/3$.
			
			Combining these estimates, we obtain that
			\[
			|\tr_n(q_j(\bX^{(n)}, \bZ^{(n)})) - \tau(q_j(\bX, \bZ))| < \varepsilon
			\]
			for $\cU$-many $n$ whenever $\bX^{(n)} \in \Gamma_{\mathbf{R}}^{(n)} ( \Tilde{\cO} \mid \bY^{(n)} \rightsquigarrow \bY)$. Thus, the claimed set inclusion in \ref{eqn: set inclusion} holds and we establish $\chi^\cU(\bX \mid \bY^{(n)} \rightsquigarrow \bY) \leq \chi^\cU(\bX \mid \bZ^{(n)} \rightsquigarrow \bZ)$ as desired.
		\end{proof}
	\end{lemma}
	
	This lemma allows us to make the following definition.
	\begin{definition} \label{def: conditional entropy embedding}
		We define the conditional microstates entropy of $\bX = (X_1, \ldots, X_m) \in (\mathcal{M}, \tau)^m_{\sa}$ given the embedding $\iota: \mathcal{B} \rightarrow \mathbb{M}_\cU$ to be
		\[
		\chi^\cU(\bX  \mid \cB, \iota) := \chi^\cU (\bX  \mid  \bY^{(n)} \rightsquigarrow \bY),
		\]
		where $\bY$ is any set of generators of $\mathcal{B}$ and $\bY^{(n)}$ is any sequence of approximating microstates such that $\iota_{\bY^{(n)} \rightsquigarrow \bY} = \iota$.
	\end{definition}
	
	This definition still depends on the choice of embedding $\iota$.  We remark that if $\cB$ is not strongly $1$-bounded in the sense of Jung \cite{ Jung2007S1B} and Hayes \cite{Hayes2018}, then there are embeddings of $\cB$ into a matrix ultraproduct that are not equivalent by an automorphism of the matrix ultraproduct \cite[Theorem 1.1, Corollary 1.4]{JekelCovering}.  However, if we are not concerned with the particular embedding, we may also consider the supremum over all embeddings.
	
	\begin{definition} \label{def: conditional entropy ultrafilter}
		We define the conditional microstates entropy of $\bX = (X_1, \ldots, X_m) \in (\mathcal{M}, \tau)^m_{\sa}$ given $\cB$, with respect to the ultrafilter $\cU$, as
		\[
		\chi^{\cU}(\bX \mid \cB) := \sup_{\iota: \cB \to \mathbb{M}_\cU} \chi^{\cU}(\bX \mid \cB, \iota).
		\]
	\end{definition}
	
	Shlyakhtenko \cite{Shlyakhtenko2002} defined a version of conditional microstates entropy, which we will denote by ${\overline{\chi}(\bX \mid \bY)}$, where the microstate $\bY^{(n)}$ is not fixed beforehand but rather chosen to maximize the volume of the microstate sets $\Gamma_{\mathbf{R}}^{(n)}(\bX \mid \bY^{(n)} \rightsquigarrow \bY)$.
	
	\begin{definition}[Shlyakhtenko \cite{Shlyakhtenko2002}]
		Let $(\cM,\tau)$ be a tracial von Neumann algebra and $\bX = (X_1,\dots,X_m)$ and $\bY = (Y_1,Y_2,\dots)$ be self-adjoint elements of $\cM$ with operator norm $\leq R$.  For a neighborhood $\cO$ of law$(\bX,\bY)$, let $\pi_2(\cO)$ be the set of laws $\lambda|_{\C\ip{y}}$ where $\lambda: \C\ip{\mathbf{x}, \, \mathbf{y}} \to \C$ is a law in $\cO$.  Define
		\[
		\overline{\chi}_\mathbf{R}(\bX \mid \bY) = \inf_{\cO \ni \operatorname{law}(\bX,\bY)} \limsup_{n \to \infty} \frac{1}{n^2} \log \left( \sup_{\bY^{(n)} \in \Gamma_\mathbf{R}^{(n)}(\pi_2(\cO))} \nu_n(\Gamma_\mathbf{R}^{(n)}(\cO \mid \bY^{(n)} \rightsquigarrow \bY)) \right) + m \log n.
		\]
		Let $\overline{\chi}(\mathbf{X} \mid \mathbf{Y})$ be the supremum over $\mathbf{R}$.
	\end{definition}
	
	\begin{remark}
		In fact, Shlyakhtenko considered only a finite tuple $\bY = (Y_1,\dots,Y_{m'})$ and wrote the definition in terms of specific neighborhoods of matrices approximating (up to some tolerance $\varepsilon > 0$) the moments of $(\bX,\bY)$ with order bounded by $\ell \in \N$.
		% given by the moments up to order $\ell$ of the matrix tuple being within a distance of $\epsilon$ of the corresponding moments of $(\bX,\bY)$.  
		For a finite tuple $\bY$, this definition agrees with what we wrote above; indeed, the quantity is monotone under inclusions of neighborhoods $\cO$, so it suffices to consider a neighborhood basis for the space of laws $\Sigma_{m+m', \mathbf{R}}$ at the point law$(\bX,\bY)$. %(here $\Sigma_{m+m',\mathbf{R}}$ means the space of laws with index set $I = \{1,\dots,m+m'\}$ and $\mathbf{R} = (R,\dots,R)$).  
		In the case where $\bY$ is an infinite tuple, a neighborhood basis could instead be given by neighborhoods that test the moments up to a certain order in each finite subset of the variables.
	\end{remark}
	
	Shlyakhtenko also showed that the entropy $\overline{\chi}_\mathbf{R}(\bX \mid \bY)$ depends only on $\bX$ and $\cB = \mathrm{W}^*(\bY)$, provided that $\mathbf{R} > \norm{\bX}_{\operatorname{op}}$.\footnote{Technically, he proved this under the assumption that $\bY$ is a finite tuple, but this does not materially affect the proof.}  Therefore, one can define $\overline{\chi}(\mathbf{X} \mid \cB)$ as the $\overline{\chi}(\mathbf{X} \mid \mathbf{Y})$ for a generating tuple $\mathbf{Y}$ of $\cB$.
	
	Here we will show that $\overline{\chi}_\mathbf{R}(\bX \mid \bY)$ represents the supremum of the entropies $\chi^{\cU}(\bX \mid \cB, \iota)$, i.e. the supremum over embeddings $\iota: \cB \rightarrow \mathbb{M}_\cU$ and over non-principal ultrafilters $\cU$.
	
	\begin{lemma} \label{lem: Shl-sup}
		Let $(\cM,\tau)$ be a tracial von Neumann algebra.  Let $\bX$ be an $m$-tuple and $\bY$ an infinite tuple with $\norm{\bX}_{\op}$ and $\norm{\bY}_{\op} \leq \mathbf{R}$.  Let $\cB = \mathrm{W}^*(\mathbf{Y})$.
		Then
		\[
		\overline{\chi}(\bX \mid \mathbf{Y}) = \overline{\chi}_\mathbf{R}(\bX \mid \bY) = \sup_{\cU} \chi^{\cU}(\bX \mid \cB) = \sup_{\cU} \sup_{\iota: \cB \to \mathbb{M}_\cU} \chi^{\cU}(\bX \mid \cB, \iota),
		\]
		where $\sup_{\cU}$ denotes the supremum over all non-principal ultrafilters on $\N$.
	\end{lemma}
	
	\begin{proof}
		($\geq$)  Fix $\cU$ and $\iota$. Let $\cO$ be a neighborhood of law$(\bX,\bY)$.  Let $\bY^{(n)}$ be a sequence of microstates for $\bY$ as in Definition \ref{def: conditional entropy embedding}.  Then
		\[
		\nu_n(\Gamma_\mathbf{R}^{(n)}(\cO \mid \bY^{(n)} \rightsquigarrow \bY)) \leq
		\sup_{\bZ^{(n)} \in \Gamma_\mathbf{R}^{(n)}(\pi_2(\cO))} \nu_n(\Gamma_\mathbf{R}^{(n)}(\cO \mid \bZ^{(n)} \rightsquigarrow \bY)).
		\]
		(In the case where the left-hand side is empty the inequality holds trivially.)  Therefore,
		\begin{multline*}
			\lim_{n \to \cU} \left[ \frac{1}{n^2} \log \nu_n(\Gamma_\mathbf{R}^{(n)}(\cO \mid \bY^{(n)} \rightsquigarrow \bY)) + m \log n \right] \\ \leq
			\limsup_{n \to \infty} \frac{1}{n^2} \log \left[ \sup_{\bZ^{(n)} \in \Gamma_\mathbf{R}^{(n)}(\pi_2(\cO))} \nu_n(\Gamma_\mathbf{R}^{(n)}(\cO \mid \bZ^{(n)} \rightsquigarrow \bY)) + m \log n \right].
		\end{multline*}
		Taking the infimum over $\cO$ and then the supremum over $\iota$ and $\cU$ completes the argument.
		
		($\leq$) Let $\cO_k$ be a nested decreasing sequence of neighborhoods of law$(\bX,\bY)$ in $\Sigma_{\omega, \mathbf{R}}$ such that ${\cO_0 = \Sigma_{\omega, \mathbf{R}}}$ and $\bigcap_{k \in \N} \cO_k = \{\text{law}(\bX,\bY)\}$.  Define $A_0 = \N$ and for $k \geq 1$,
		\[
		A_k = \left\{n \geq k: \frac{1}{n^2} \log \sup_{\bZ^{(n)} \in \Gamma_\mathbf{R}^{(n)}(\pi_2(\cO_k))} \nu_n(\Gamma_\mathbf{R}^{(n)}(\cO_k\mid \bZ^{(n)} \rightsquigarrow \bY)) + m \log n > \overline{\chi}_\mathbf{R}(\bX \mid \bY) - \frac{1}{k} \right\}
		\]
		Observe that because $\cO_{k+1} \subseteq \cO_k$, we have $A_{k+1} \subseteq A_k$.  It is a consequence of the definition of $\overline{\chi}_\mathbf{R}(\bX \mid \bY)$ that $A_k$ is nonempty.  However, $\bigcap_{k \in \N} A_k = \varnothing$.  These properties imply that there exists a non-principal ultrafilter $\cU$ on $\N$ such that $A_k \in \cU$ for all $k$ (see Fact \ref{fact: IIP ultrafilters exist}).
		
		From the definition of $A_k$, for each $n \in A_k \setminus A_{k+1}$, there exists $\bY^{(n)} \in \Gamma_\mathbf{R}^{(n)}(\pi_2(\cO_k))$ such that
		\[
		\frac{1}{n^2} \log \nu_n(\Gamma_\mathbf{R}^{(n)}(\cO_k\mid \bY^{(n)} \rightsquigarrow \bY)) + m \log n > \overline{\chi}_\mathbf{R}(\bX \mid \bY) - \frac{1}{k}.
		\]
		
		Note that law$(\bY^{(n)})$ converges to law$(\bY)$ as $n \to \cU$ because for $n \in A_k$ (which is a neighborhood of $\cU$), we have law$(\bY^{(n)}) \in \pi_2(\cO_k)$; and $\bigcap_{k \in \N} \pi_2(\cO_k) = \{\text{law}(\bY)\}$.  Thus, $\bY^{(n)}$ induces an embedding $\iota: \cB \to \mathbb{M}_\cU$.  For $n \in A_k$, we have $n \in A_{k'} \setminus A_{k'+1}$ for some $k' \geq k$ and hence
		\begin{align*}
			\frac{1}{n^2} \log \nu_n(\Gamma_\mathbf{R}^{(n)}(\cO_k\mid \bY^{(n)} \rightsquigarrow \bY)) + m \log n &\geq \frac{1}{n^2} \log \nu_n(\Gamma_\mathbf{R}^{(n)}(\cO_{k'}\mid \bY^{(n)} \rightsquigarrow \bY)) + m \log n \\
			&> \overline{\chi}_\mathbf{R}(\bX \mid \bY) - \frac{1}{k'} \\
			&\geq \overline{\chi}_\mathbf{R}(\bX \mid \bY) - \frac{1}{k}.
		\end{align*}
		
		It follows that for each $k \in \N$,
		\[
		\lim_{n \to \cU} \frac{1}{n^2} \log \nu_n(\Gamma_\mathbf{R}^{(n)}(\cO_k\mid \bY^{(n)} \rightsquigarrow \bY)) + m \log n \geq \overline{\chi}_\mathbf{R}(\bX \mid \bY) - \frac{1}{k}.
		\]
		Because $\Sigma_{\omega, \mathbf{R}}$ is compact, the $\cO_k$'s must form a neighborhood basis for $\Sigma_{\omega, \mathbf{R}}$ at law$(\bX,\bY)$ and hence
		\begin{multline*}
			\inf_{\cO \ni \text{law}(\bX,\bY)} \frac{1}{n^2} \log \nu_n(\Gamma_\mathbf{R}^{(n)}(\cO \mid \bY^{(n)} \rightsquigarrow \bY)) + m \log n \\
			= \inf_{k \in \N} \lim_{n \to \cU} \frac{1}{n^2} \log \nu_n(\Gamma_\mathbf{R}^{(n)}(\cO_k\mid \bY^{(n)} \rightsquigarrow \bY)) + m \log n \geq \overline{\chi}_\mathbf{R}(\bX \mid \bY).
		\end{multline*}
		Therefore,
		\[
		\overline{\chi}_\mathbf{R} (\bX \mid \bY) \leq \chi^{\cU}(\bX \mid \cB, \iota),
		\]
		which completes the proof.
	\end{proof}
	
	\subsection{Conditional free entropy and conditional classical entropy} \label{subsec: conditional microstates}
	
	In this section, we relate conditional free entropy with classical entropy by an analogue of \cite[Proposition B.7]{ST2022}, which expresses the free entropy as the supremum of limits of classical entropy of certain random matrix models.  The idea of free entropy as a large-$n$ limit of classical entropy goes back to the work of Voiculescu \cite{VoiculescuFE2}, and the idea to take the supremum of classical entropies of certain measures was applied by Biane and Dabrowski in their concavification of free entropy \cite[see Remark 4.5]{BD2013}. 
	The point of \cite[Proposition B.7]{ST2022} and the results here is to give a description that does not explicitly reference the microstate spaces, based on a diagonalization argument.  We will first give a random matrix characterization of $\chi^{\cU}(\bX \mid \cB, \iota)$ and then a random matrix characterization of $\chi^{\cU}(\bX \mid \cB)$.
	
	\begin{theorem}[Random matrix characterization of $\chi^\mathcal{U}(\bX \mid  \bY)$] \label{thm: rand mx interpretation of chi^U(X | Y)}
		
		Let $\bX$ be an $m$-tuple of elements from $(\cM, \tau)_{\sa}$, and let $\bY$ be a self-adjoint tuple of generators for a $*$-subalgebra $\mathcal{B} \subseteq \cM$. Fix a sequence $\bY^{(n)} \in M_n(\C)_{\sa}^\N$ such that $\sup_n \norm{Y_j^{(n)}} < \infty$ and $\mathbf{Y}^{(n)}$ converges in non-commutative law to $\mathbf{Y}$. \\
		Then $\chi^\cU (\bX \mid  \bY^{(n)} \rightsquigarrow \bY)$ is the supremum of
		\[
		\lim_{n \rightarrow \cU} \frac{1}{n^2} h(\bX^{(n)}) + m \log n = \lim_{n \rightarrow \cU} h^{(n)}(\bX^{(n)}),
		\]
		over all sequences of random variables $(\bX^{(n)})_{n \in \N}$ with $\bX^{(n)} \in M_{n}(\C)^m_{\sa}$ satisfying:
		\begin{enumerate}
			\item[(1)] the law of $(\bX^{(n)}, \bY^{(n)})$ converges to the law of $(\bX, \bY)$ in probability as $n \rightarrow \cU$.
			\item[(2)] For some $\bR \in (0, \infty)^m$, we have $\lim_{n \rightarrow \cU} \norm{\bX^{(n)}}_{\op} \leq \bR$ in probability.
			\item[(3)] There are some constants $C > 0$ and $K > 0$ such that for each $n \in \N$,
			\[
			\mathbb{P}(\norm{\bX^{(n)}}_2 \geq C + \delta)  \leq e^{-K n^2 \delta^2} \text{ for all } \delta > 0.
			\]
			%Note we can choose $\norm{\bY^{(n)}}$ uniformly bounded by 1, so we don't have to worry about this bound for the $\bY^{(n)}$. 
		\end{enumerate}
	\end{theorem}
	
	\begin{remark}
		``Convergence in probability as $n \to \cU$'' is understood as follows.  Let $Z_n$ be a sequence of random variables indexed by $n \in \N$.  Then $\lim_{n \to \cU} Z_n \leq c$ in probability means that for each $\epsilon > 0$, $\lim_{n \to \cU} P(Z_n > c + \epsilon) = 0$.
	\end{remark}
	
	\begin{remark}
		The hypothesis (3) can be replaced by some weaker tail bounds if desired.  The motivation for the specific form of the tail bound used here is that it arises naturally from concentration estimates for certain random matrix models (see e.g.\ \cite[\S 2.3 and \S 4.4]{AGZ2009}).  For instance, it is easy to check that the Gaussian random matrices $S_j^{(n)}$ satisfy such an estimate.   This is how the theorem was used in \cite[Appendix B]{ST2022}.  However, for the purposes of our main result in this paper, the specific form of tail bound is irrelevant because we only need to apply the theorem to matrices $\bX^{(n)}$ that are uniformly bounded in operator norm, hence also uniformly bounded in $\norm{\cdot}_2$, which trivially satisfy (3); see Remark \ref{rem: uniformly bounded random matrix models} and proof of Theorem \ref{thm: cond main}.
	\end{remark}

	\begin{proof}[Proof of Theorem \ref{thm: rand mx interpretation of chi^U(X | Y)}]
		Based on Lemma \ref{lemma: cond microstates f.e. embedding}, we may assume without loss of generality that $\norm{Y_j}_{\op} = 1$ and $\norm{Y_j^{(n)}}_{\op} = 1$ for all $j \in \N$. 
		Suppose that $(\bX^{(n)})_{n \in \N}$ is any sequence satisfying the conditions (1)-(3). As the $\bY^{(n)}$ are fixed, the distribution of $(\bX^{(n)}, \bY^{(n)})$ depends only on $\bX^{(n)}$. Then one can use the argument of \cite[Proposition B.7]{ST2022}  (replacing the usual microstate spaces with the conditional microstate spaces and the limit with the ultralimit), to show that $\lim_{n \rightarrow \cU} h^{(n)}(\bX^{(n)}) \leq \chi^\cU(\bX \mid  \bY^{(n)} \rightsquigarrow \bY)$. 
		We repeat the argument here for the reader's convenience. 
		
		First fix $\mathbf{R}'$ so that for all $j = 1, \ldots, m$, $R'_j > \max(R_j, 1)$. Let $\mathcal{O}$ be a neighborhood of $\text{law}(\bX, \bY)$ in $\Sigma_{\omega, \bR'}$.
		For each $n \in \N$, we define the following partition of $M_{n}(\C)^m_{\sa}$:
		\begin{align*}
			S_0^{(n)} &:= \Gamma_{\bR'}^{(n)}( \cO \mid  \bY^{(n)} \rightsquigarrow \bY) \\
			S_1^{(n)} &:= B(0, C+1) \setminus S_0^{(n)} \\
			S_j^{(n)} &:= B(0, C + j) \setminus B(0, C + j - 1) \text{ for all $j \geq 2$},
		\end{align*}
		where $B(0, K)$ denotes the 2-norm ball of radius $K$ centered at $0$.
		To simplify notation, for each $n \in \N$ we write $\nu_{n}$ for the Lebesgue measure on ${M_{n}(\C)^m_{\sa}}$.
		Then, applying Fact \ref{fact: entropy controlled by ptn} to the density of ${(\bX^{(n)}, \, \bY^{(n)})}$, we have 
		\begin{align*}
			h_{M_{n}(\C)^m_{\sa}} (\bX^{(n)}, \,  \bY^{(n)}) 
			% &=  \int_{M_{n}(\C)^m_{\sa}} \rho_{(\bX^{(n)}, \, \bY^{(n)})} \log \rho_{(\bX^{(n)}, \, \bY^{(n)})} d \nu_n   \\
			&\leq  \sum_{j=0}^\infty \mu^{(n)}(S_j^{(n)}) \log \nu_n (S_j^{(n)}) - \mu^{(n)} (S_j^{(n)}) \log \mu^{(n)} (S_j^{(n)}),
		\end{align*}
		where we write $\mu^{(n)}$ to denote the probability measure associated to $(\bX^{(n)}, \, \bY^{(n)})$. (Note that as $\bY^{(n)}$ is fixed, this depends only on $\bX^{(n)}$).
		Subsequently,
		\begin{equation} \label{eqn: ptn-bound}
			h^{(n)} (\bX^{(n)}, \,  \bY^{(n)}) \leq 
			\sum_{j=0}^\infty  H_j^{(n)} (\cO \mid  \bY^{(n)} \rightsquigarrow \bY), 
		\end{equation}
		where
		\[
		H_j^{(n)}(\cO \mid  \bY^{(n)} \rightsquigarrow \bY) : = \mu^{(n)}(S_j^{(n)}) \left( \frac{1}{n^2} \log \nu_n (S_j^{(n)}) + m \log n \right) - \mu^{(n)}(S_j^{(n)}) \cdot \frac{1}{n^2} \log \mu^{(n)}(S_j^{(n)}).
		\]
		Note that this quantity depends on the neighborhood $\cO$ only for the indices $j \in \{0, 1\}$. 
		
		We now bound each term of the sum in \ref{eqn: ptn-bound}. When $j = 0$, we have
		\begin{align*}
			H_0^{(n)}(\cO \mid  \bY^{(n)} \rightsquigarrow \bY) &= \mu^{(n)}( \Gamma_{\bR'}^{(n)}( \cO \mid  \bY^{(n)} \rightsquigarrow \bY) )
			\left( \frac{1}{n^2} \log \nu_n (\Gamma_{\bR'}^{(n)}( \cO \mid  \bY^{(n)} \rightsquigarrow \bY)) + m \log n \right) \\
			&\quad - \mu^{(n)}(\Gamma_{\bR'}^{(n)}( \cO \mid  \bY^{(n)} \rightsquigarrow \bY)) \left( \frac{1}{n^2} \log \mu^{(n)}(\Gamma_{\bR'}^{(n)}( \cO \mid  \bY^{(n)} \rightsquigarrow \bY)) \right).
		\end{align*}
		Since $(\bX^{(n)}, \bY^{(n)})$ converges in non-commutative law to $(\bX, \bY)$ and $-t \log t \leq e^{-1}$ for any $t > 0$, we have that the second term goes to $0$ when we take the ultralimit in $n$.  For the first term, note that by property (1),
		\[
		\lim_{n \rightarrow \cU} \mu^{(n)} ( \Gamma^{(n)}_{\bR'}(\cO \mid  \bY^{(n)} \rightsquigarrow \bY) ) = 1.
		\]
		Then, by the definition of $\chi^\cU(\bX \mid  \bY^{(n)} \rightsquigarrow \bY)$: for any $\epsilon > 0$ there is some neighborhood $\cO_\epsilon$ such that 
		\begin{equation} \label{eqn: bound for H_0}
			\lim_{n \rightarrow \cU}  H_0^{(n)} (\cO_\epsilon \mid  \bY^{(n)} \rightsquigarrow \bY)   \leq \chi^\cU (\bX \mid  \bY^{(n)} \rightsquigarrow \bY) + \epsilon.
		\end{equation}
		For bounding the other terms, we first note that upon identifying $M_{n}(\C)^m_{\sa}$ with $\C^{m n^2}$ and applying Stirling's formula, we have
		\begin{equation} \label{eqn: stirling estimate}
			\frac{1}{n^2} \log \nu_n ( B(0, r)) = - m \log n + m \log r + O(m) \text{ for all $r > 0$}.
		\end{equation}
		Then for the term with $j = 1$, note that $S_1^{(n)} \subseteq B(0, C+1)$, so we have
		\begin{align*}
			H_1^{(n)}(\cO \mid  \bY^{(n)} \rightsquigarrow \bY) &\leq  \mu^{(n)}(S_1^{(n)}) \left( - m \log n + m \log (C + 1) + O(m) + m \log n \right) -  \frac{1}{n^2} \mu^{(n)}(S_1^{(n)}) \log  \mu^{(n)}(S_1^{(n)}) \\
			&= \mu^{(n)}(S_1^{(n)}) \left(m \log (C + 1) + O(m) \right) -  \frac{1}{n^2} \mu^{(n)}(S_1^{(n)}) \log  \mu^{(n)}(S_1^{(n)}).
		\end{align*}
		Again since $-t \log t \leq e^{-1}$ for $t >0$, the second term goes to zero in the ultralimit.
		Then again by property (1), we have $\mu^{(n)}(S_1^{(n)}) \rightarrow 0$ as $n \rightarrow \cU$, so that:
		\[
		\lim_{n \rightarrow \cU}  H_1^{(n)} (\cO \mid  \bY^{(n)} \rightarrow \bY) \leq \lim_{n \rightarrow \cU} \mu^{(n)}(S_1^{(n)}) (m \log(C + 1) + O(m)) = 0.
		\]
		Finally, for $j \geq 2$, we apply equation (\ref{eqn: stirling estimate}) along with the exponential tail bound in property (3) and the fact that $-t \log t$ is increasing for $t \leq e^{-1}$ to obtain
		\begin{align*}
			\sum_{j=2}^\infty H_j^{(n)} (\cO \mid  \bY^{(n)} \rightsquigarrow \bY)  &\leq \sum_{j=2}^\infty  \mu^{(n)}(S_j^{(n)}) (m \log (C+j) + O(m)) + \frac{1}{n^2} \int - e^{-1} \log(e^{-1})  \\
			&\leq  \sum_{j=2}^\infty e^{-K n^2 (j-1)^2} (m \log (C+j) + O(m)) + \frac{1}{e n^2} \rightarrow 0 \text{ as $n \rightarrow \cU$}.
		\end{align*}
		Finally, combining these upper bounds and sending $\epsilon \rightarrow 0$ in equation (\ref{eqn: bound for H_0}), we obtain that for any sequence $(\bX^{(n)})_{n \in \N}$ satisfying the conditions (1)-(3),
		\[
		\lim_{n \rightarrow \cU}   h^{(n)}(\bX^{(n)}, \,  \bY^{(n)})  \leq \lim_{n \rightarrow \cU} \sum_{j=0}^\infty    H_j^{(n)} (\cO \mid  \bY^{(n)} \rightsquigarrow \bY) 
		\leq \chi^\cU (\bX \mid  \bY^{(n)} \rightsquigarrow \bY).
		\]

		For the other direction, we construct a sequence $(\bX^{(n)})_{n \in \N}$ so that
		\[
		\chi^\cU (\bX \mid  \bY^{(n)} \rightsquigarrow \bY) \leq \lim_{n \rightarrow \cU} h^{(n)}(\bX^{(n)}).
		\]
		Without loss of generality, assume $\chi^\cU( \bX \mid  \bY^{(n)} \rightsquigarrow \bY) > - \infty$. Fix $\bR > \norm{\bX}_{\infty}$, and $\cO$ a neighborhood of law$(\bX, \bY)$.
		Let $(\cO_k)_{k \in \N}$ be a sequence of nested neighborhoods of law$(\bX, \bY)$ in $\Sigma_{\omega, \bR}$, shrinking to law$(\bX, \bY)$ as $k \rightarrow \infty$.
		
		Let $A_0 = \N$ and for $k \geq 1$ let
		\[
		A_k = \biggl\{n \geq k: \frac{1}{n^2} \log \nu_n(\Gamma_\bR^{(n)}(\cO_k \mid \mathbf{Y}^{(n)} \rightsquigarrow \mathbf{Y})) + m \log n
		> \chi_\bR^\cU (\bX \mid  \bY^{(n)} \rightsquigarrow \bY) - \frac{1}{k} \biggr\}.
		\]
		Note that $A_k \in \cU$ because
		\[
		\lim_{n \to \cU} \frac{1}{n^2} \log \nu_n(\Gamma_\bR^{(n)}(\cO_k \mid \mathbf{Y}^{(n)} \rightsquigarrow \mathbf{Y})) + m \log n \geq \chi_\bR^\cU (\bX \mid  \bY^{(n)} \rightsquigarrow \bY).
		\]
		Moreover, since the $\cO_k$'s are nested, we have $A_{k+1} \subseteq A_k$; also, $\bigcap_{k \in \N} A_k = \varnothing$.  For each $k$, for $n \in A_k \setminus A_{k+1}$, let $\mu^{(n)}$ be the uniform measure on $\Gamma_\bR^{(n)} (\cO_k \mid  \bY^{(n)} \rightsquigarrow \bY)$, and let $\bX^{(n)}$ be a random matrix tuple in $M_n(\C)^m_{\sa}$ with distribution $\mu^{(n)}$.  Thus, for $n \in A_k \setminus A_{k+1},$
		\[
		h^{(n)}(\bX^{(n)}) = \frac{1}{n^2} \log \nu_n(\Gamma_\bR^{(n)}(\cO_k \mid \mathbf{Y}^{(n)} \rightsquigarrow \mathbf{Y})) + m \log n.
		\]
		Hence, using the definition and nestedness of the $A_k$'s,
		\[
		h^{(n)}(\bX^{(n)}) \geq \chi_\bR^\cU (\bX \mid  \bY^{(n)} \rightsquigarrow \bY) - \frac{1}{k} \text{ for } n \in A_k.
		\]
		Thus, as $A_k \in \cU$, we have
		\[
		\lim_{n \to \cU} h^{(n)}(\bX^{(n)}) \geq \chi_\bR^{\cU}(\bX \mid \bY^{(n)} \rightsquigarrow \bY) - \frac{1}{k}.
		\]
		Since $k$ was arbitrary,
		\[
		\lim_{n \to \cU} h^{(n)}(\bX^{(n)}) \geq \chi_\bR^{\cU}(\bX \mid \bY^{(n)} \rightsquigarrow \bY) = \chi^{\cU}(\bX \mid \bY^{(n)} \rightsquigarrow \bY).
		\]
		
		It remains to check that $\bX^{(n)}$ satisfies (1)-(3).  Note that by construction $\norm{\bX^{(n)}}_{\op} \leq \bR$ so that (2) and (3) hold.  Furthermore, to show that law$(\bX^{(n)},\bY^{(n)})$ converges in probability to law$(\bX,\bY)$, fix a neighborhood $\cO$ of law$(\bX,\bY)$.  Since $\Sigma_{d, \bR}$ is compact, there exists $k$ such that $\cO_k \subseteq \cO$. For all $n \in A_k$, we have law$(\bX^{(n)},\bY^{(n)}) \in \cO_k \subseteq \cO$.
	\end{proof}
	
	\begin{remark}
		The same proof works to show the random matrix interpretation of $\chi(\bX \mid  \bY)$ holds with $\lim_{n \rightarrow \cU}$ replaced by subsequential limits.
		More explicitly, one can establish that given fixed microstates $\bY^{(n)}$ for $\bY$, we have $\chi(\bX \mid  \bY^{(n)} \rightsquigarrow \bY)$ is the supremum of
		\[
		\lim_{\ell \rightarrow \infty} \frac{1}{n_\ell^2} h(\bX^{(\ell)}, \bY^{(n_\ell)}) + m \log n_\ell = \lim_{\ell \rightarrow \infty} h^{(n_\ell)}(\bX^{(\ell)}, \bY^{(n_\ell)}),
		\]
		over sequences $n_\ell \to \infty$ and random variables $\bX^{(\ell)} \in M_{n_\ell}(\C)^m_{\sa}$ satisfying (1)-(3).
	\end{remark}

	\begin{remark} \label{rem: uniformly bounded random matrix models}
		In the last theorem, note for each $n$ we chose the matrix tuples $\bX^{(n)}$ uniformly at random from $\Gamma_\bR^{(n)} (\cO_k \mid  \bY^{(n)} \rightsquigarrow \bY)$ for $n \in A_k \setminus A_{k+1}$. (Recall $\bigcup_{k \in \N} (A_k \setminus A_{k+1}) = A_0 = \N$.) Thus, we may always choose the collection $\{\bX^{(n)}\}_{n \in \N}$ to be unitarily invariant in distribution and uniformly bounded in operator norm.
	\end{remark}

	\begin{theorem} \label{thm: rand mx conditional entropy 2}
		Let $\bX$ be an $m$-tuple of elements from $(\cM, \tau)_{\sa}$, and let $\bY$ be a self-adjoint tuple of generators for a $*$-subalgebra $\mathcal{B} \subseteq \cM$.  Then $\chi^\cU (\bX \mid  \cB)$ is the supremum of
		\[
		\lim_{n \to \cU} \frac{1}{n^2} h(\bX^{(n)} \mid  \bY^{(n)} ) + m \log n = \lim_{n \rightarrow \cU} h^{(n)}(\bX^{(n)} \mid  \bY^{(n)}),
		\]
		over all sequences of random variables $(\bX^{(n)},\bY^{(n)})_{n \in \N}$ with $\bX^{(n)} \in M_n(\C)^m_{\sa}$ and $\bY^{(n)} \in M_n(\C)_{\sa}^{\N}$ satisfying:
		\begin{enumerate}
			\item[(1)] The law of $(\bX^{(n)}, \bY^{(n)})$ converges to the law of $(\bX, \bY)$ in probability as $n \to \cU$.
			\item[(2)] For some $\bR \in (0, \infty)^m$, we have $\lim_{n \rightarrow \cU} \norm{\bX_j^{(n)}}_{\op} \leq \bR$ in probability.
			\item[(3)] There are some constants $C > 0$ and $K > 0$ such that for each $n \in \N$ and all values of $y$,
			\[
			\mathbb{P}(\norm{\bX^{(n)}}_2 \geq C + \delta  \mid  \bY^{(n)} = y)  \leq e^{-K n^2 \delta^2} \text{ for all } \delta > 0.
			\]
			\item[(4)] For each $j$, there is some $R_j'$ such that $\lim_{n \to \cU} \norm{Y_j^{(n)}}_{\operatorname{op}} \leq R_j'$.
		\end{enumerate}
	\end{theorem}
	
	\begin{proof}
		First, we show that $\chi^{\cU}(\bX \mid \cB)$ is less than or equal to the supremum.  Recall $\chi^{\cU}(\bX \mid \cB)$ was defined as the supremum of $\chi^{\cU}(\bX \mid \cB, \iota)$ over all embeddings $\iota$.  By the previous proposition, after fixing microstates $\bY^{(n)}$ that induce the embedding $\iota$, this quantity in turn is the supremum over matrix models $\bX^{(n)}$ satisfying (1)-(3) of Theorem \ref{thm: rand mx interpretation of chi^U(X | Y)}. Note that if $\bX^{(n)}$ and $\bY^{(n)}$ satisfy (1)-(3) of the previous theorem, where $\bY^{(n)}$ is deterministic, then $(\bX^{(n)},\bY^{(n)})$ also satisfy (1)-(4) of this theorem.  Moreover, when $\bY^{(n)}$ is deterministic, we have $h^{(n)}(\bX^{(n)} \mid \bY^{(n)}) = h^{(n)}(\bX^{(n)})$.  Thus, $\chi^{\cU}(\bX \mid \cB)$ is less than or equal to the supremum asserted in this theorem.
		
		For the reverse inequality, suppose $(\bX^{(n)},\bY^{(n)})$ satisfy (1)-(4) of this theorem.  By standard facts about conditional distributions and Borel probability spaces (see e.g. \cite[Chapters V.9, V.10]{feller1991introduction} and \cite[\S 8.3]{Klenke2014}), we may assume without loss of generality that all the $\bY^{(n)}$'s are random variables on a fixed probability space $(\Omega, \mathbb{P}_0)$, and that $\bX^{(n)}$ is a random variable on $M_n(\C)_{\sa}^m \times \Omega$ which can be sampled by first sampling $\omega$ and then sampling using the conditional distribution $\mu_\omega^{(n)}$ of $\bX^{(n)}$ given $\bY^{(n)}(\omega)$.  For each $\omega$, we denote by $\bX^{(n)}(\omega)$ the random variable on $M_n(\C)$ chosen according to this conditional distribution.  We use $\mathbb{P}_0$ for the probability measure on the space where $\mathbf{Y}^{(n)}$ is defined, as above, and we write $\mathbb{P}$ for probability on the (implicit) larger probability space where $\mathbf{X}^{(n)}$ and $\mathbf{Y}^{(n)}$ live.
		
		Since $h^{(n)}(\bX^{(n)} \mid \bY^{(n)}) = \int_\Omega h^{(n)}(\bX^{(n)}(\omega))\,d\mathbb{P}_0(\omega)$, the idea of the argument is now to apply Theorem \ref{thm: rand mx interpretation of chi^U(X | Y)} pointwise to each $\omega$.  However, we must be careful because we only assumed convergence in probability, and the limits are with respect to the ultrafilter $\cU$.  Thus, roughly speaking, we want to arrange that the hypotheses of Theorem \ref{thm: rand mx interpretation of chi^U(X | Y)} hold uniformly for $\omega$ in a set of large measure.
		
		Note condition (3) implies an upper bound on the second moment $E \norm{\bX^{(n)}(\omega)}_2^2$ that is uniform in $\omega$; specifically,
		\begin{align*}
			\E \norm{X^{(n)}(\omega)}_2^2 &= \sum_{j=1}^{\infty} \norm{X^{(n)}(\omega)}_2^2 \cdot \mathbf{1}_{\{C+j \leq \norm{X^{(n)}(\omega)}_2^2 < C + j + 1) \, \mid  \, \bY^{(n)}(\omega) \}} + \norm{X^{(n)}(\omega)}_2^2 \cdot \mathbf{1}_{\{ \norm{X^{(n)}(\omega)}_2^2 < C+1 \, \mid \, \bY^{(n)}(\omega) \}} \\
			&\leq \sum_{j=1}^\infty (C + j + 1) e^{-K n^2 j^2} + (C+1) =: C' < \infty,
		\end{align*}
		and this upper bound $C'$ is independent of $\omega$.  Hence, by Fact \ref{fact: entropy controlled by var}, there is an upper bound $M$ on the entropy $h^{(n)}(\bX^{(n)}(\omega))$ that is uniform in $\omega$. 
		
		Let $(\cO_k)$ be a sequence of neighborhoods in $\Sigma_{\omega,\bR,\bR'}$ shrinking to $\operatorname{law}(\bX,\bY)$.  For each $n$ and $\omega$, write
		\[
		\Gamma_{n,k}(\omega) = \{\bX \in M_n(\C)_{\sa}^m: \norm{\bX}_\infty < R_j + 1/k \text{ and } \operatorname{law}(\bX,\bY^{(n)}(\omega)) \in \cO_k \}.
		\]
		Let
		\begin{align*}
			\Omega_{n,k} = \{\omega \in \Omega: & \norm{Y_j(\omega)}_\infty < R_j' + 1/k \text{ for } j = 1, \dots k, \\
			& \mu_\omega^{(n)}\bigl(\Gamma_{n,k}(\omega)) > 1 - 1/k \}.
		\end{align*}
		Note that $\lim_{n \to \cU} \mathbb{P}_0(\Omega_{n,k}) = 1$; indeed,
		\begin{align*}
			\mathbb{P}_0(\Omega_{n,k}^c) &\leq \mathbb{P}_0(\norm{Y_j(\omega)}_\infty > R_j + 1/k) + \mathbb{P}_0(\mu_\omega^{(n)}\bigl(\Gamma_{n,k}(\omega)^c) \geq 1/k) \\
			&\leq 
			\mathbb{P}_0(\norm{Y_j(\omega)}_\infty > R_j + 1/k) + k \int_\Omega \mu_\omega^{(n)}\bigl(\Gamma_{n,k}(\omega)^c) \,d\mathbb{P}_0,
		\end{align*}
		where we have used the union bound and Markov's inequality.  Now $\mathbb{P}_0(\norm{Y_j(\omega)}_\infty > R_j + 1/k) \to 0$ as $n \to \cU$ by assumption (4).  Meanwhile, by the definition of the conditional distribution,
		\[
		\int_{\Omega} \mu_\omega^{(n)}\bigl(\Gamma_{n,k}(\omega)^c) \,d\mathbb{P}_0 = \mathbb{P}(\{\bX \in M_n(\C)_{\sa}^m: \norm{\bX}_\infty < R_j + 1/k \text{ and } \operatorname{law}(\bX,\bY^{(n)}(\omega)) \in \cO_k \}^c),
		\]
		which converges to zero as $n \to \cU$ by our assumptions (1) and (2).  Let $\Omega_{0,k} = \Omega$ and $A_0 = \N$ and for $k \geq 1$, let
		\[
		A_k = \{n \geq k: \mathbb{P}_0(\Omega_{n,k}) \geq 1 - 1/k \}.
		\]
		Since $\mathbb{P}_0(\Omega_{n,k}) \to 1$ as $n \to \cU$, we have that $A_k \in \cU$.  Also the intersection of the $A_k$'s is empty by construction.  For each $k \in \N$, for each $n \in A_k \setminus A_{k+1}$, fix some $\omega_n \in \Omega_{n,k}$ such that
		\[
		h^{(n)}(\bX^{(n)}(\omega_n)) \geq \sup_{\omega \in \Omega_{n,k}} h^{(n)}(\bX^{(n)}(\omega)) - 1/n.
		\]
		From the definition of $\Omega_{n,k}$, we see that for $n \in A_k \setminus A_{k+1}$, we have
		\[
		\norm{Y_j(\omega_n)}_\infty < R_j' + 1/k \text{ for } j = 1, \dots k, \qquad  \text{and } \qquad \mu_{\omega_n}^{(n)}\bigl(\Gamma_{n,k}(\omega_n)) > 1 - 1/k.
		\]
		Since the $A_k$'s are nested, it follows that the same inequality holds for all $n \in A_k$.  Therefore, by the definition of $\Gamma_{n,k}$, we get that $\bX^{(n)}(\omega_n)$ and $\bY^{(n)}(\omega_n)$ satisfy the assumptions of Theorem \ref{thm: rand mx interpretation of chi^U(X | Y)}, and so
		\[
		\lim_{n \to \cU} h^{(n)}(\bX^{(n)}(\omega_n)) \leq \chi^{\cU}(\bX \mid \cB).
		\]
		Now observe by the definition of conditional entropy, if $n \in A_k \setminus A_{k+1}$, then
		\begin{align*}
			h^{(n)}(\bX^{(n)} \mid \bY^{(n)}) &= \int_\Omega h^{(n)}(\bX^{(n)}(\omega))\,d\mathbb{P}_0(\omega) \\
			&= \int_{\Omega_{n,k}} h^{(n)}(\bX^{(n)}(\omega)) \,d\mathbb{P}_0(\omega) + \int_{\Omega_{n,k}^c} h^{(n)}(\bX^{(n)}(\omega)) \,d\mathbb{P}_0(\omega) \\
			&\leq \mathbb{P}_0(\Omega_{n,k}) (h^{(n)}(\bX^{(n)}(\omega_n)) + 1/n) + \mathbb{P}_0(\Omega_{n,k}^c) M \\
			&\leq h^{(n)}(\bX^{(n)}(\omega_n)) + 1/n + M/k.
		\end{align*}
		Again, because the $A_k$'s are nested, we get $h^{(n)}(\bX^{(n)} \mid \mathbf{Y}^{(n)}) \leq h^{(n)}(\bX^{(n)}(\omega_n)) + 1/n + M/k$ for all $n \in A_k$, not only for $n \in A_k \setminus A_{k+1}$.  Hence,
		\[
		\lim_{n \to \cU} h^{(n)}(\bX^{(n)} \mid \bY^{(n)}) \leq \lim_{n \to \cU} h^{(n)}(\bX^{(n)}(\omega_n)) + 1/k + M/k \leq \chi^{\cU}(\bX \mid \cB) + 1/k + M/k.
		\]
		Since $k$ was arbitrary, we get $\lim_{n \to \cU} h^{(n)}(\bX^{(n)} \mid \bY^{(n)}) \leq \chi^{\cU}(\bX \mid \cB)$ as desired.
	\end{proof}
	
	\section{Conditional Non-Microstates Free Entropy} \label{sec: conditional non-microstates}
	
	In this section, we develop the notions of free score function and free Fisher information in order to define a conditional non-microstates free entropy, in analogue with the classical conditional entropy defined using Fisher information in Section \ref{sec: cond fisher info}.
	The following definitions are from \cite{VoiculescuFE5}, but phrased in terms of formal polynomials rather than von Neumann subalgebras with algebraically free generators.
	
	\begin{notation}
		Throughout this section, we fix $\vec x$ to denote the indeterminates $x_1, x_2, \ldots x_m$, and $\vec y$ to denote the infinite tuple of indeterminates $y_1, y_2, \ldots$.
	\end{notation}
	
	\begin{definition}
		Let $\cB$ be a unital $*$-algebra, and
		let $\cB \scal{\vec x}$ be the algebra of polynomials in $m$ non-commuting variables with coefficients from $\cB$. For $i = 1, \ldots, m$, we define the \emph{partial non-commutative derivatives $\partial_i$} or \emph{free difference quotient} as linear maps:
		\[
		\partial_i : \cB \scal{\vec x} \rightarrow \cB \scal{\vec x} \otimes \cB \scal{\vec x}
		\]
		by sending $\cB$ to $0$ and $x_j$ to $\delta_{ij} 1 \otimes 1$, and extending by the Leibniz rule
		% \[
		%     \partial_i b = 0, \qquad \partial_i x_j = \delta_{ij} 1 \otimes 1 \quad (\text{for } j = 1, \ldots, m), \text{ and the Leibniz rule}
		% \]
		\[
		\partial_i(P_1 P_2) = \partial_i(P_1) \cdot (1 \otimes P_2)  + (P_1 \otimes 1) \cdot \partial_i(P_2) \qquad (\text{for } P_1, P_2 \in \cB \scal{\vec x}).
		\]
		% In particular, on monomials we have
		% \[
		%     \partial_i(b x_{i(1)} \cdots x_{i(\ell)}) = b \left( \sum_{k=1}^\ell \delta_{i, i(k)} x_{i(1)} \cdots x_{i(k-1)} \otimes x_{i(k+1)} \cdots x_{i(\ell)} \right).
		% \]
		We also extend the non-commutative derivatives $\partial_i$ for $i = 1, \ldots, m$ to the algebra of polynomials in countably many non-commuting variables $\cB \scal{\vec x, \vec y}.$ 
	\end{definition}
	
	As motivation for the free score function, we recall the following relationship between the free difference quotient and differentiation of functions on the matrix algebra.  In Lemma \ref{lem: NC polynomial derivative}, Lemma \ref{lem: NC polynomial divergence}, and Corollary \ref{cor: IBP matrix version}, the subalgebra $\cB$ is absent. Instead, we use an infinite collection of parameters $Y_1$, $Y_2$, \dots which we will take to be matrix approximations for generators of $\cB$. (In spirit, we are replacing $\cB \scal{\vec x}$ with $\C \scal{\vec x, Y_1, Y_2, \ldots}$ when $Y_1$, $Y_2$, \dots are self-adjoint generators of $\cB$ and the $x_j$'s are formal variables; see Remark \ref{rmk: C(x,y) vs B(x)}).  
	In the next two lemmas, we use the following notation:  For an algebra $A$, let $\#: (A \otimes A) \times A \to A$ be the bilinear map $(a \otimes b,c) \mapsto acb =: (a \otimes b) \# c$.
	
	\begin{lemma} \label{lem: NC polynomial derivative}
		Let $f \in \C \scal{\vec x, \vec y}$.  Let $f^{(n)}$ denote the function $M_n(\C)_{\sa}^{\N} \to M_n(\C)$ given by evaluation of $f$.  Let $D_j f^{(n)}$ denote the Fr\'echet derivative of $f^{(n)}$ as a function of $X_j$ mapping $M_n(\C)_{\sa} \to M_n(\C) \cong M_n(\C)_{\sa} \otimes_{\R} \C$, that is, $D_j f^{(n)}(X_1,\dots,X_m,Y_1,Y_2,\dots)$ is the linear transformation $M_n(\C) \to M_n(\C)$ such that for $A \in M_n(\C)_{\sa}$,
		\[
		D_j f^{(n)}(X_1,\dots,X_m,Y_1,Y_2,\dots)[A] = \frac{d}{dt} \bigr|_{t=0} f^{(n)}(X_1,\dots,X_{j-1},X_j+tA,X_{j+1},\dots,X_m,Y_1,Y_2,\dots).
		\]
		Then we have
		\[
		D_j f^{(n)}(X_1,\dots,X_m,Y_1,Y_2,\dots)[A] = (\partial_j f)^{(n)}(X_1,\dots,X_m,Y_1,Y_2,\dots) \# A,
		\]
		where $(\partial_j f)^{(n)}$ is function $M_n(\C)_{\sa}^{\N} \to M_n(\C) \otimes M_n(\C)$ given by evaluation of both tensorands of $\partial_j f$ as polynomials on the input matrices $X_1$, \dots, $X_m$, $Y_1$, $Y_2$, \dots.
	\end{lemma}
	
	This lemma is proved for monomials by direct computation and then extended by linearity.  For details, see e.g. \cite[Lemma 14.1.3]{JekelThesis}.  Similar statements were shown earlier in \cite{Cebron2013} and \cite[\S 3]{DHK2013}.
	
	\begin{lemma} \label{lem: NC polynomial divergence}
		Let $f \in \C\scal{\vec x, \vec y}$.  Then $\frac{1}{n^2} \nabla_j \cdot f^{(n)} = \tr_n \otimes \tr_n((\partial_j f)^{(n)})$, where $\nabla_j \cdot$ denotes the divergence operator on $f^{(n)}$ as a function of $X_j$.
	\end{lemma}
	
	\begin{proof}
		Recall that the divergence is the trace of the Fr\'echet derivative.  Note that there is an isomorphism $\Phi$ from $M_n(\C) \otimes M_n(\C)^{\operatorname{op}}$ to the space of linear operators $B(M_n(\C))$ given by $\Phi(A \otimes B)[C] = (A \otimes B) \# C$.  Because of the uniqueness of the trace on $M_n(\C) \otimes M_n(\C)^{\operatorname{op}} \cong M_{n^2}(\C)$, it follows that $\Tr_{B(M_n(\C))}[\Phi(T)] = \Tr_n \otimes \Tr_n(T)$ for $T \in M_n(\C) \otimes M_n(\C)^{\operatorname{op}}$.  Therefore,
		\[
		\nabla_j \cdot f^{(n)} = \Tr_{B(M_n(\C))}(D_j f^{(n)}) = \Tr_n \otimes \Tr_n((\partial_j f)^{(n)}).
		\]
		Dividing by $n^2$ proves the asserted formula.
	\end{proof}
	
	This yields the following integration-by-parts formula for random matrix models.  This is closely related to the Schwinger-Dyson equation for free Gibbs laws; see e.g.\ \cite[\S 4.3.20]{GuionnetParkCity}.
	
	\begin{corollary} \label{cor: IBP matrix version}
		Let $\mathbf{X}$ be an $M_n(\C)_{\sa}^m$-valued random variable with finite moments, and let $\Xi$ be a score function for $\mathbf{X}$.  Let $Y_1$, $Y_2$, \dots be deterministic self-adjoint matrices.  Let $f_1, \dots, f_m \in \C\scal{\vec x, \vec y}$ and $f = (f_1,\dots,f_m)$.  Then
		\[
		\frac{1}{n^2} \E \scal{\Xi_j, \, f^{(n)}(X_1,\dots,X_m,Y_1,Y_2,\dots)}_{\tr_n} = \E \tr_n \otimes \tr_n [(\partial_j f)^{(n)}(X_1,\dots,X_m,Y_1,Y_2,\dots)]
		\]
	\end{corollary}
	
	\begin{proof}
		Because $f^{(n)}$ is a polynomial function, Lemma \ref{lem: poly growth IBP} shows that
		\[
		\E \scal{\Xi, f^{(n)}(X_1,\dots,X_m,Y_1,Y_2,\dots) }_{\tr_n} = \E \nabla_X \cdot f^{(n)}(X_1,\dots,X_m,Y_1,Y_2,\dots)
		\]
		Note that $\nabla_X \cdot f^{(n)}$, the divergence with respect to $X$, can be expressed as $\sum_{j=1}^m \nabla_j \cdot f_j^{(n)}$.  Finally, we apply Lemma \ref{lem: NC polynomial divergence} to evaluate $\nabla_j \cdot f_j^{(n)}$.
	\end{proof}
	
	Voiculescu's free score function is defined to satisfy an analog of this formula on $\cB$-valued non-commutative polynomials.

	% For the definitions that follow, recall that $L^1(\cM)$ is the set of trace-class operators from $\cM$, i.e. the set of all $X \in \cM$ so that $\tau(|X|) < \infty$. Also recall that $L^2(\cM)$ is the usual Gelfand-Naimark-Segal representation of $\cM$, a Hilbert space constructed using $\cM$ and $\tau$ so that $\cM$ acts naturally via left multiplication and the inner product is defined by $\scal{x,y} = \tau(y^*x)$. (For further details on this construction, see e.g. \cite[Section 2.6]{AP2017}).
	
	\begin{definition}
		Let $\mathcal{B}$ be a unital $*$-subalgebra of $(\mathcal{M}, \tau)$, and $\bX = (X_1, \ldots, X_m)$ an $m$-tuple of self-adjoint elements in $(\mathcal{M}, \tau)$. 
		Let $\cB \scal{\bX} \subset \cM$ denote the subalgebra generated by $\cB$ and $\bX$.
		An element $\xi = (\xi_1, \ldots, \xi_m) \in L^2(\mathrm{W}^*(\cB \scal{\bX}))^m$ is the \emph{free score function of $\bX$ with respect to $\mathcal{B}$} (also called the conjugate variable of $\bX$ with respect to $\cB$), if
		\[
		\tau(\xi_i p (\bX)) = \tau \otimes \tau( \partial_i p(\bX)) \text{ for all } p(x) \in \cB\scal{\vec x} \text{ and } 1 \leq i \leq m.
		\]
		Such an element $\xi$ will be denoted $\mathscr{J}(\bX: \cB)$.
	\end{definition}

	\begin{definition} \label{def: cond free Fisher info}                
		Let $\mathcal{B} \subset (\mathcal{M}, \tau)$ be a unital $*$-subalgebra and $\bX = (X_1, \ldots, X_m)$ an $m$-tuple of self-adjoint elements of $\mathcal{M}$. 
		We define the \emph{relative free Fisher information of $\bX$ with respect to $\mathcal{B}$} to be
		\[
		\Phi^*(\bX: \mathcal{B}) = 
		\norm{\mathscr{J}(\bX: \cB \scal{\bX})}_2^2 = 
		\sum_{1 \leq j \leq m} \norm{\mathscr{J}(X_j: \cB \scal{\bX}) }_2^2.
		\]
		Otherwise, we define $\Phi^*(\bX: \cB) = \infty$. We will also use the notation $\Phi^*(\bX: \bY)$ when $\bY$ is a self-adjoint tuple which generates $\cB$.
		% If $\cB = \C$, we write $\Phi^*(\bX)$ instead of $\Phi^*(\bX \mid  \C)$ and call it the \emph{free information} of $\bX$.
	\end{definition}
	
	\begin{lemma} \label{lemma: cond tau-char-of-phi*}
		Let $\bX = (X_1, X_2, \ldots, X_m)$ be a tuple from $(\cM, \tau)_{\sa}$. Then
		\[
		\Phi^*(\bX : \cB) = \sup \{ | \tau \otimes \tau (\partial f (\bX))|^2 \ : \ f(x) \in \cB\scal{\vec x}^m, \ \norm{f(\bX)}_2 \leq 1 \}.
		\]
	\end{lemma}

	\begin{proof}
		First, if both sides are infinite, then there is nothing to show; it suffices to consider the cases where either $\Phi^*(\bX :  \cB)$ or the right-hand side above are finite.
		
		\noindent 
		If $\Phi^*(\bX: \cB) < \infty$, then by Definition \ref{def: cond free Fisher info},
		there is a conjugate variable $\xi := \mathscr{J}(\bX: \cB\scal{\bX}) \in L^2(\cM)^m$ satisfying
		\[
		\scal{\xi, f(\bX)}_\tau = \tau \otimes \tau(\partial f(\bX)) \text{ for all } f(x) \in \cB \scal{\vec x}^m.
		\]
		The result then follows:
		\begin{align*}
			\Phi^*(\bX :  \cB) = \norm{\xi}_{2}^2 &= \sup \{| \scal{\xi, f(\bX)}_\tau|^2 \ : \ \norm{f(\bX)}_2 \leq 1\} \\
			&= \sup \{ |\tau \otimes \tau (\partial f(\bX)) |^2 \ : \ \norm{f(\bX)}_2 \leq 1 \}.
		\end{align*}
		
		\noindent 
		On the other hand, if the supremum in the right-hand side above is finite, then letting $C < \infty$ be such that
		\[
		\sup \{ | \tau \otimes \tau (\partial f (\bX))|^2 \ : \ f(x) \in \cB \scal{\vec x}^m, \ \norm{f(\bX)}_{2} \leq 1 \} \leq C < \infty,
		\]
		then we have for all polynomials $f(x) \in \cB \scal{\vec x}^m$,
		\[
		| \tau \otimes \tau (\partial f(\bX)) | \leq C \norm{f(\bX)}_2.
		\]
		Because non-commutative polynomials are dense in $L^2(\cM)$, there is a bounded linear functional $\phi: L^2(\cM) \to \C$ extending $f(\bX) \mapsto \tau \otimes \tau(\partial f(\bX))$, and by the Riesz representation theorem, there is a vector $\xi \in L^2(\cM)$ satisfying 
		\[
		\scal{\xi, f(\bX)}_\tau = \phi(f(\bX)) = \tau \otimes \tau (\partial f(\bX)) \text{ for all } f(x) \in \cB \scal{\vec x}^m.
		\]
		Then $\xi$ is a conjugate variable for $\bX$ and by definition $\Phi^*(\bX :  \cB) < \infty$, so the equality follows by the argument given above.
	\end{proof}
	
	\begin{remark} \label{rmk: C(x,y) vs B(x)}
		The supremum on the right-hand side of Lemma \ref{lemma: cond tau-char-of-phi*} can be rephrased in terms of formal polynomials $p(\vec x, \vec y) \in \C \scal{\vec x, \vec y}$. 
		
		Indeed, since $\mathrm{W}^*(\bY) = \cB$, given any $f(\vec x) \in \cB \scal{\vec x}^m$, and any $\cB$-valued coefficient $c$ of $f$, we can choose a sequence of polynomials $p^c_n(\vec y) \in \C \scal{\vec y}^m$ so that $\norm{c - p^c_n(\bY)}_2 \rightarrow 0.$
		Also by the Kaplansky density theorem (similar to Lemma \ref{lem: polynomial approximation}) we can choose these to be uniformly bounded in operator norm, in particular $\norm{p^c_n(\bY)}_{\op} \leq \norm{c}_{\op}$.
		By taking the sum of polynomial approximations of the $\cB$-valued coefficients of monomials in $f(\vec x)$, we obtain polynomial approximations $p^f_n(\vec x, \bY)$ satisfying
		\[
		\norm{f(\bX) - p^f_n(\bX, \bY)}_2 \rightarrow 0,
		\]
		with $\norm{p^f_n(\bX, \bY)}_{\op}$ uniformly bounded. 
		%(In particular by $\text{len}(f) \cdot \max\{ \norm{c} \ : \ c \text{ is a coefficient appearing in $f$}\}$.   
		
		Now write $f(\vec x) = \lim_n p^f_n(\vec x, \bY)$, and note 
		\begin{align*}
			\sup_{\substack{ f(\vec x) \in \cB \scal{\vec x}^m \\ \norm{f}_2 \leq 1}} \left\{ \left| \tau \otimes \tau \left( \partial f (\bX) \right) \right|^2 \right\} &= \sup_{\substack{ f(\vec x) \in \cB \scal{\vec x}^m \\ \norm{f}_2 \leq 1}} \left\{ \left| \tau \otimes \tau \left( \partial (\lim_n p^f_n(\bX, \bY)) \right) \right|^2 \right\} \\
			&= \sup_{\substack{ f(\vec x) \in \cB \scal{\vec x}^m \\ \norm{f}_2 \leq 1}} \left\{ \left| \tau \otimes \tau \left( \lim_n \partial p^f_n(\bX, \bY) \right) \right|^2 \right\} \\
			&\leq \sup_{\substack{ p(\vec x, \vec y) \in \C \scal{\vec x, \vec y}^m \\ \norm{p}_2 \leq 1}} \left\{ \left| \tau \otimes \tau \left( \partial p (\bX, \bY) \right) \right|^2 \right\},
		\end{align*}
		where we can exchange the limit with the derivative in the second equality because for each of the $\cB$-valued coefficients $c$ in the polynomial $f$, $p_n^c(\vec y)$ is uniformly bounded in operator norm and converges in $\norm{\cdot}_2$ to $c$.  The other inequality is immediate as there is a natural identification of $\{p(\vec x, \bY) \ : \ p(\vec x, \vec y) \in \C \scal{ \vec x, \vec y}\}$ as elements of $\mathcal{B} \scal{\vec x}$ (by considering any $\bY$-terms as coefficients).
		Thus, we have that  
		\[
		\sup_{\substack{ f(x) \in \cB \scal{\vec x}^m \\ \norm{f}_2 \leq 1}} \left\{ \left| \tau \otimes \tau \left( \partial f (\bX) \right) \right|^2 \right\} = \sup_{\substack{ p(x, y) \in \C \scal{\vec x, \vec y}^m \\ \norm{p}_2 \leq 1}} \left\{ \left| \tau \otimes \tau \left( \partial p (\bX, \bY) \right) \right|^2 \right\}.
		\]
	\end{remark}

	\begin{definition} \label{def: relative nm free ent}
		Let $\cB \subseteq \cM$ be a unital $*$-subalgebra generated by the tuple of self-adjoint elements $\bY$, and let $\bX$ an $m$-tuple of self-adjoint elements of $\cM$. We define the \emph{relative non-microstates free entropy of $\bX$ with respect to $\cB$} as
		\[
		\chi^*(\bX  :   \cB) = \frac{1}{2} \int_0^\infty \left( \frac{m}{1 + t} - \Phi^*(\bX + t^{1/2} \mathbf{S} :  \cB) \right) \, dt + \frac{m}{2} \log 2\pi e,
		\]
		where $\mathbf{S} = (S_1, \ldots, S_m)$ is an $m$-tuple of freely independent standard semicircular variables which are also free from $\cB\scal{\bX}$. We also write $\chi^*(\bX :  \bY)$ for this quantity.
	\end{definition}

	\section{Proof of the Main Result} \label{sec: main proof}

	The next step toward proving the main result in Theorem \ref{thm: cond main} is the following proposition relating the free and the classical Fisher information.  The intuition for this statement is related to \cite[\S 4.6]{ShlyakhtenkoParkCity}.
	
	\begin{proposition} \label{prop: cond Phi* upper bound}
		Let $\bX = (X_1, \ldots, X_m)$ be an $m$-tuple of self-adjoint elements in $(\cM, \tau)$, and fix $\bY$ a tuple of generators for a von Neumann subalgebra $\mathcal{B} \subseteq \cM$. 
		Fix a sequence $\bY^{(n)} \in M_n(\C)^\N_{\sa}$ of deterministic microstates for $\bY$ such that $\norm{Y_j^{(n)}}$ is uniformly bounded for each $j$.
		Suppose that ${\bX^{(n)} \in M_{n}(\C)^{m}_{\sa}}$ is a sequence of random matrix tuples with finite moments such that the law of $(\bX^{(n)},\bY^{(n)})$ converges in probability to the law of $(\bX,\bY)$ and for each $k$, we have  \\
		${\lim_{n \to \cU} \E \tr_n((X_j^{(n)})^{2k}) < \infty}$ and $\lim_{n \to \cU} \tr_n((Y_j^{(n)})^{2k}) < \infty$. 
		Then
		\begin{equation} \label{eq: cond Phi* upper bound}
			\Phi^*(\bX :  \bY) \leq \lim_{n \rightarrow \cU} \frac{1}{n^4}  \ \mathcal{I}(\bX^{(n)}) = \lim_{n \to \cU} \mathcal{I}^{(n)} (\bX^{(n)}).
		\end{equation}
	\end{proposition}
	
	Our first step toward proving this is to observe that
	\[
	\lim_{n \to \cU} \E[\tr_n(p(\bX^{(n)},\bY^{(n)})] = \tau(p(\bX,\bY)).
	\]
	Here we use the following standard fact from probability theory.
	
	\begin{lemma} \label{lem: uniform integrability}
		Suppose $Z_n$ is a sequence of complex random variables and $Z_n$ converges in probability to a constant $c$.  If $\lim_{n \to \cU} \norm{Z_n}_{L^2} < \infty$, then $\lim_{n \to \cU} \mathbb{E}[Z_n] = c$.
	\end{lemma}
	
	\begin{proof}
		Note that $Z_n - c$ is also bounded in $L^2$ and converges to zero in probability.  Hence, we can assume without loss of generality that $c = 0$.  Note that for $0 < \delta < t$,
		\begin{align*}
			|\E Z_n| \leq \mathbb{E} |Z_n| &= \mathbb{E} \left( |Z_n| 1_{|Z_n| < \delta} \right) + \E\left( |Z_n| 1_{\delta \leq |Z_n| \leq t} \right) + \E \left( |Z_n| 1_{|Z_n| > t} \right) \\
			&\leq \delta + t P(|Z_n| \geq \delta) + \E \left( \frac{|Z_n|^2}{t} 1_{|Z_n| \geq t} \right) \\
			&\leq \delta + t P(|Z_n| \geq \delta) + \frac{\norm{Z_n}_{L^2}^2}{t}.
		\end{align*}
		Hence,
		\[
		\lim_{n \to \cU} |\E Z_n| \leq \delta + 0 + \frac{1}{t} \lim_{n \to \cU} \norm{Z_n}_{L^2}^2.
		\]
		Since $\delta$ and $t$ were arbitrary, $\lim_{n \to \cU} |\E Z_n| = 0$.
	\end{proof}
	
	\begin{proof}[Proof of Proposition \ref{prop: cond Phi* upper bound}]
		For readability, we write $\vec x = (x_1, \ldots, x_m)$ and $\vec y = (y_1, y_2, \ldots)$. We claim that for every $p \in \C \ip{\vec x, \vec y}$, we have
		\begin{equation} \label{eq: convergence of expectation 1}
			\lim_{n \to \cU} \E[\tr_n(p(\bX^{(n)},\bY^{(n)})] = \tau(p(\bX,\bY)).
		\end{equation}
		and
		\begin{equation} \label{eq: convergence of expectation 2}
			\lim_{n \to \cU} \E[\tr_n \otimes \tr_n(\partial_{x_i} p(\bX^{(n)},\bY^{(n)}))] = \tau \otimes \tau(\partial_{x_i} p(\bX,\bY)).
		\end{equation}
		By our assumptions, $\tr_n(p(\bX^{(n)},\bY^{(n)}))$ converges in probability to $\tau(p(\bX,\bY))$ and $\tr_n \otimes \tr_n(\partial_{x_i} p(\bX^{(n)},\bY^{(n)}))$ converges in probability to $\tau \otimes \tau(\partial_{x_i} p(\bX,\bY))$.  Hence, by Lemma \ref{lem: uniform integrability}, it suffices to show that $\tr_n(p(\bX^{(n)},\bY^{(n)}))$ and $\tr_n \otimes \tr_n(\partial_{x_i} p(\bX^{(n)},\bY^{(n)}))$ are bounded in $L^2$ of the probability space as $n \to \cU$.  By linearity, it suffices to prove that $\tr_n(p(\mathbf{X}^{(n)},\bY^{(n)}))$ and $\tr_n(p(\bX^{(n)},\bY^{(n)}) \tr_n(q(\bX^{(n)},\bY^{(n)}))$ are bounded in $L^2$ for monomials $p$ and $q$.  Consider a monomial $A_1^{(n)} \dots A_k^{(n)}$ where each $A_j$ is either $X_{i_j}^{(n)}$ or $Y_{i_j}^{(n)}$.  By the non-commutative H\"older's inequality (Fact \ref{fact: Holder}),
		\[
		|\tr_n(A_1^{(n)} \dots A_k^{(n)})|^2 \leq (\tr_n(|A_1^{(n)}|^k)^{2/k} \dots \tr_n(|A_k^{(n)}|^k)^{2/k})^2 \leq \tr_n((A_1^{(n)})^{2k})^{1/k} \dots \tr_n((A_k^{(n)})^{2k})^{1/k}.
		\]
		Then using the classical H\"older's inequality,
		\begin{align*}
			\E |\tr_n(A_1^{(n)} \dots A_k^{(n)})|^2 &\leq \E \left[ \tr_n((A_1^{(n)})^{2k})^{1/k} \dots \tr_n((A_k^{(n)})^{2k})^{1/k} \right] \\
			&\leq \left( \E \tr_n((A_1^{(n)})^{2k}) \right)^{1/k} \dots \left( \E \tr_n((A_k^{(n)})^{2k}) \right)^{1/k}.
		\end{align*}
		By our assumption on the moments, the limit of this quantity as $n \to \cU$ is finite.  Similarly, if $q$ is a monomial $B_1^{(n)} \cdots B_\ell^{(n)}$ where each $B_j^{(n)}$ is one of the $X^{(n)}$'s or $Y^{(n)}$'s, then
		\begin{align*}
			\E |&\tr_n(A_1^{(n)} \dots A_k^{(n)}) \tr_n(B_1^{(n)} \dots B_\ell^{(n)})|^2 \\
			&\leq \E \left[ \tr_n((A_1^{(n)})^{2k})^{1/k} \dots \tr_n((A_k^{(n)})^{2k})^{1/k} \tr_n((B_1^{(n)})^{2\ell})^{1/\ell} \dots \tr_n((B_\ell^{(n)})^{2\ell})^{1/\ell} \right] \\
			&\leq \left( \E \tr_n((A_1^{(n)})^{2k})^2 \right)^{1/2k} \dots \left( \E \tr_n((A_k^{(n)})^{2k})^2 \right)^{1/2k}  \left( \E \tr_n((B_1^{(n)})^{2\ell})^2 \right)^{1/2\ell} \dots \left( \E \tr_n((B_\ell^{(n)})^{2\ell})^2 \right)^{1/2\ell} \\
			&\leq \left( \E \tr_n((A_1^{(n)})^{4k}) \right)^{1/2k} \dots \left( \E \tr_n((A_k^{(n)})^{4k}) \right)^{1/2k}  \left( \E \tr_n((B_1^{(n)})^{4\ell}) \right)^{1/2\ell} \dots \left( \E \tr_n((B_\ell^{(n)})^{4\ell}) \right)^{1/2\ell}
		\end{align*}
		This establishes \eqref{eq: convergence of expectation 1} and \eqref{eq: convergence of expectation 2}.
		
		The main inequality (\ref{eq: cond Phi* upper bound}) that we want to prove is trivial in the case that $\lim_{n \to \cU} \frac{1}{n^4} \mathcal{I}(\bX^{(n)})) = \infty$, so assume that the limit is finite.  Then for $\cU$-many $n$ there is a score function $\Xi^{(n)}$ for $\bX^{(n)}$ as in Definition \ref{def: classical score function}.  Then by Corollary \ref{cor: IBP matrix version}, we have
		\begin{equation} \label{eqn: IBP matrix version}
			\E \scal{\Xi^{(n)}, p(\bX^{(n)},\bY^{(n)})}_{\tr_n} 
			%= \E \left[ \operatorname{div}_X p(\bX^{(n)},\bY^{(n)}) \right] 
			= \sum_{j=1}^m \E \left[ n^2 \ \tr_{n} \otimes \tr_{n} \left( \partial_{x_j} p \left( \bX^{(n)}, \bY^{(n)} \right) \right) \right].
		\end{equation}
		By Lemma \ref{lemma: cond tau-char-of-phi*} and Remark \ref{rmk: C(x,y) vs B(x)},
		\[
		\Phi^*(\bX :  \bY) = \sup_{\substack{ f(x) \in \cB \scal{\vec x}^m \\ \norm{f}_2 \leq 1}} \left\{ \left| \tau \otimes \tau \left( \partial f (\bX) \right) \right|^2 \right\} = \sup_{\substack{ p(x, y) \in \C \scal{\vec x, \vec y}^m \\ \norm{p(\bX,\bY)}_2 \leq 1}} \left\{ \left| \tau \otimes \tau \left( \partial p (\bX, \bY) \right) \right|^2 \right\}.
		\]   
		Fix $p \in \C\ip{\vec x, \vec y}$ with $\norm{p(\bX,\bY)}_2 \leq 1$.  Then
		\[
		\tau \otimes \tau(\partial p(\bX,\bY)) = \lim_{n \to \cU} \mathbb{E}[\tr_n \otimes \tr_n(\partial p(\bX^{(n)},\bY^{(n)}))].
		\]
		% Using integration by parts:
		For each $n$, by applying the integration by parts relation (\ref{eqn: IBP matrix version}) and Cauchy-Schwarz inequality,
		\begin{align*}
			\left| \E \left[ \tr_{n} \otimes \tr_{n} \left( \partial_X p \left( \bX^{(n)}, \bY^{(n)} \right) \right) \right] \right| &= \left| \frac{1}{n^2} \E \scal{\Xi^{(n)}, p(\bX^{(n)}, \bY^{(n)})}_{\tr_n} \right| \\
			&\leq \frac{1}{n^2} \ [\E \norm{p(\bX^{(n)}, \bY^{(n)})}_2^2 \  \norm{\Xi^{(n)}}_2^2]^{1/2} \\
			&= \frac{1}{n^2} \ (\E \norm{p(\bX^{(n)}, \bY^{(n)})}_2^2)^{1/2} \ \mathcal{I}(\bX^{(n)})^{1/2}.
		\end{align*}
		Thus, in the limit we obtain
		\begin{align*}
			|\tau \otimes \tau(\partial p(\bX,\bY))| &= \lim_{n \to \cU} |\mathbb{E}[\tr_n \otimes \tr_n(\partial p(\bX^{(n)},\bY^{(n)}))] | \\
			&\leq \lim_{n \to \cU} (\E \norm{p(\bX^{(n)},\bY^{(n)})}_2^2 )^{1/2} \left( \frac{1}{n^4} \mathcal{I}(\bX^{(n)}) \right)^{1/2} \\
			&= \left( \norm{p(\bX,\bY)}_2^2 \lim_{n \to \cU} \frac{1}{n^4} \mathcal{I}(\bX^{(n)}) \right)^{1/2} \\
			&\leq \left( \lim_{n \to \cU} \frac{1}{n^4} \mathcal{I}(\bX^{(n)}) \right)^{1/2}
		\end{align*}
		By squaring this equation and taking the supremum over $p$ with $\norm{p(\bX,\bY)}_2 \leq 1$, we obtain \\ 
		${\Phi^*(\bX : \bY) \leq \lim_{n \to \cU} \frac{1}{n^4} \mathcal{I}(\bX^{(n)})}$ as desired.
	\end{proof}
	
	Since we must integrate the Fisher information to obtain the entropy, we will use the following fact to exchange the ultralimit and integral (which is difficult to do in general).
	
	\begin{lemma} \label{lem: swap ultralimit and integral}
		Let $f_n: [0,T] \to [0,\infty]$ be a sequence of decreasing nonnegative functions for some fixed $0 < T < \infty$.  Let $\cU$ be a free ultrafilter on $\N$.  Then $f = \lim_{n \to \cU} f_n$ is a decreasing function (hence measurable) and $\int_0^T f \leq \lim_{n \to \cU} \int_0^T f_n$.
	\end{lemma}
	
	\begin{proof}
		Let
		\[
		g_k(t) = \sum_{j=1}^{2^k} \min\left( f \left( \frac{Tj}{2^k} \right), 2^k \right) \mathbf{1}_{[T(j-1)/2^k,Tj/2^k)}(t).
		\]
		Then $g_k \leq g_{k+1} \leq 2^{k+1}$ and $\lim_{k \to \infty} g_k = f$ almost everywhere (specifically, at all points of continuity of $f$, which is cocountable because $f$ is monotone).  Hence, $\int_0^T f = \sup_k \int_0^T g_k$.  Fix $k \in \N$ and $\epsilon > 0$.  Since $f(t) = \lim_{n \to \cU} f_n(t)$, there is a set $A \in \mathcal{U}$ such that
		\[
		n \in A \implies f_n \left( \frac{Tj}{2^k} \right) \geq \min\left( f \left( \frac{Tj}{2^k} \right), 2^k \right) - \epsilon \text{ for } j = 1, \dots, 2^k.
		\]
		Since $f_n$ is decreasing, $\int_0^T f_n$ is bounded by below by any right-endpoint Riemann sum, so for $n \in A$,
		\[
		\int_0^T f_n(t)\,dt \geq \frac{T}{2^k} \sum_{j=1}^{2^k} f_n \left( \frac{Tj}{2^k} \right) \geq \frac{T}{2^k} \sum_{j=1}^{2k} \left( \min\left( f_n \left( \frac{Tj}{2^k} \right), 2^k \right) - \epsilon \right) = \int_0^T g_k(t)\,dt - T \epsilon. 
		\]
		Since $\epsilon$ and $k$ were arbitrary, $\lim_{n \to \cU} \int_0^T f_n \geq \sup_k \int_0^T g_k = \int_0^T f$.
	\end{proof}

	\begin{theorem} \label{thm: cond main}
		Let $\bX = (X_1, \ldots, X_m)$ be an $m$-tuple of self-adjoint elements in $(\cM, \tau)$, let $\cB$ be a separable von Neumann subalgebra, and fix a countable tuple of generators $\bY$ for $\mathcal{B}$. Fix a sequence $\bY^{(n)} \in M_n(\C)^\N_{\sa}$ converging in non-commutative law to $\bY$. Then for any ultrafilter $\cU$ on $\N$:
		\begin{equation} \label{eq: first main theorem}
			\chi^\cU(\bX \mid  \bY^{(n)} \rightsquigarrow \bY) \leq \chi^*(\bX :  \cB).
		\end{equation}
		Hence, by taking the supremum over all sequences $\bY^{(n)} \rightsquigarrow \bY$, and the supremum over ultrafilters,
		\begin{equation} \label{eq: second main theorem}
			\chi^{\cU}(\bX \mid \cB) \leq \chi^*(\bX : \cB), \qquad \overline{\chi}(\bX \mid \cB) \leq \chi^*(\bX : \cB).
		\end{equation}
	\end{theorem}
	
	\begin{proof}
		First, note that if $\chi^{\cU}(\bX \mid \cB) = -\infty$, there is nothing to prove.  Otherwise, by Theorem \ref{thm: rand mx interpretation of chi^U(X | Y)} and Remark \ref{rem: uniformly bounded random matrix models}, there is some $\bR \in (0, \infty)^m$ and a sequence $\bX^{(n)}$ of random variables in $M_{n}(\C)^{m}_{\sa}$ such that:
		\begin{itemize}
			\item $\norm{\bX^{(n)}}_{\operatorname{op}} \leq \bR$ for all $j$ and $n$, 
			\item $(\bX^{(n)}, \bY^{(n)})$ converges to $(\bX, \bY)$ in non-commutative law, and
			\item $ \displaystyle
			\chi^{\cU}(\mathbf{X} \mid \mathbf{Y}^{(n)} \rightsquigarrow \mathbf{Y}) = \lim_{n \to \cU} \left( \frac{1}{n^2} h(\bX^{(n)}) + m \log n \right) = \lim_{n \to \cU} h^{(n)}(\bX^{(n)}). $
		\end{itemize}

		Let $\mathbf{S} = (S_1, S_2, \ldots, S_m)$ be a tuple of freely independent semicircular random variables, also freely independent from $\bX$ and $\bY$, in some larger tracial von Neumann algebra $\tilde{\cM}$ containing $\cM$. 
		Also, let $\mathbf{S}^{(n)} \in M_{n}(\C)^m_{\sa}$ be a sequence of $m$-tuples of independent GUE random matrices.  We claim that for each $u > 0$, the random matrices $\bX^{(n)} + u^{1/2} \mathbf{S}^{(n)}$ and $\mathbf{Y}^{(n)}$ satisfy the hypotheses of Proposition \ref{prop: cond Phi* upper bound}.  Indeed, using Theorem \ref{thm: asymptotic free independence} (b), $(\mathbf{S}^{(n)},\bX^{(n)},\mathbf{Y}^{(n)})$ converges in non-commutative law almost surely to $(\mathbf{S},\mathbf{X},\mathbf{Y})$. 
		Since any polynomial in $(\mathbf{X} + u^{1/2} \mathbf{S},\mathbf{Y})$ is also a polynomial in $(\mathbf{S},\mathbf{X},\mathbf{Y})$, this implies that $(\bX^{(n)} + u^{1/2} \mathbf{S}^{(n)},\mathbf{Y}^{(n)})$ converges in non-commutative law almost surely to $(\mathbf{X} + u^{1/2} \mathbf{S},\mathbf{Y})$ .  
		Furthermore, using Theorem \ref{thm: asymptotic free independence} (a) and the fact that $\norm{\bX^{(n)}}_{\operatorname{op}} \leq \bR$, we have $\lim_{n \to \cU} \E \norm{X_j^{(n)} + u^{1/2} S_j^{(n)}}_{\operatorname{op}}^k < \infty$ for each $k$.  This implies $\lim_{n \to \cU} \E \tr_n[(X_j^{(n)} + u^{1/2} S_j^{(n)})^{2k}] < \infty$.  Therefore, we can apply Proposition \ref{prop: cond Phi* upper bound} to $(\bX^{(n)} + u^{1/2} \mathbf{S}^{(n)},\mathbf{Y}^{(n)})$ to obtain
		\[
		\Phi^*(\bX + u^{1/2} \mathbf{S} :  \bY ) \leq \lim_{n \rightarrow \cU} \frac{1}{n^4} \  \mathcal{I}(\bX^{(n)} + u^{1/2} \mathbf{S}^{(n)}).
		\]
		
		% Taking an integral over $u$ on $[0,t]$ and applying Lemma \ref{lem: swap ultralimit and integral} to the sequence of decreasing functions $\{\mathcal{I}(\bX^{(n)} + u^{1/2} \mathbf{S}^{(n)})\}$ (each restricted to $[0,t]$), we have
		By Corollary \ref{cor: Fisher info decreasing}, the function $\mathcal{I}(\bX^{(n)} + u^{1/2} \mathbf{S}^{(n)})$ is decreasing in $u$. By restricting to a bounded interval $[0,t]$, integrating over $u$ in this interval, and applying Lemma \ref{lem: swap ultralimit and integral}, we obtain
		\begin{align*}
			\frac{1}{2} \int_0^t \Phi^*(\bX + u^{1/2} \mathbf{S}  :  \bY ) \, du &\leq \lim_{n \rightarrow \cU} \frac{1}{2 n^4} \int_0^t \mathcal{I}^{(n)}(\bX^{(n)} + u^{1/2} \mathbf{S}^{(n)} ) \, du \\
			&= \lim_{n \to \cU} \frac{1}{n^4} \left(h^{(n)}(\bX^{(n)} + t^{1/2} \mathbf{S}^{(n)}) - h^{(n)}(\bX^{(n)}) \right),
		\end{align*}
		where the last equality follows from \eqref{eqn: increments of h is integral of I()}.  Rearranging this to isolate the term approximating ${\chi^{\cU}(\mathbf{X} \mid \cB)}$, we write
		\begin{align*}
			\chi^{\cU}(\mathbf{X} \mid \mathbf{Y}^{(n)} \rightsquigarrow \mathbf{Y}) &= \lim_{n \to \cU}  h^{(n)}(\bX^{(n)})  \\
			&\leq \lim_{n \to \cU} h^{(n)}(\mathbf{X}^{(n)} + t^{1/2} \mathbf{S}^{(n)}) - \frac{1}{2} \int_0^t \Phi^*(\bX + u^{1/2} \mathbf{S} :  \mathcal{B}) \, du.
		\end{align*}
		We may bound the first term on the right-hand side by applying Fact \ref{fact: entropy controlled by var} to $\bX^{(n)} + t^{1/2} \mathbf{S}^{(n)}$: 
		% \begin{align} \label{eqn: cond bound for h(Xs)}
			%     h(\bX^{(n)} + t^{1/2} \mathbf{S}^{(n)})
			%     &\leq \frac{m n^2}{2} \log \left( \frac{2 \pi e}{m n^2} \int \norm{\bX^{(n)} + t^{1/2} \mathbf{S}^{(n)}}^2  d \mu_{\bX^{(n)} + t^{1/2} \mathbf{S}^{(n)}} \right) \notag \\
			%     &\leq \frac{m n^2}{2} \left( \log( E \norm{\bX^{(n)}}_2^2 + m \cdot t) + \log \left( \frac{2\pi e}{m n^2} \right) \right).
			% \end{align}
		% Then, we may bound the first term by using (\ref{eqn: cond bound for h(Xs)}):
		\begin{align*}
			h^{(n)}(\bX^{(n)} + t^{1/2} \mathbf{S}^{(n)}) &= \frac{1}{n^2} h(\bX^{(n)} + t^{1/2} \mathbf{S}^{(n)}) + m \log n \\
			&\leq \frac{m}{2} \left( \log(\E \norm{\bX^{(n)}}_2^2 + m \cdot t) + \log\left( \frac{2\pi e}{m n^2} \right) + 2 \log n \right) \\
			&= \frac{m}{2} \log  \left( \frac{(\E \norm{\bX^{(n)}}_2^2 + m \cdot t)(2\pi e)(n^2)}{m n^2} \right) \\
			&\leq \frac{m}{2} \log \left( C + t \right) + \frac{m}{2} \log(2\pi e),
		\end{align*}
		where $C := \frac{1}{m} \sup_n \E \norm{\mathbf{X}^{(n)}}_2^2 < \infty$.  Therefore, 
		\begin{align*}
			\chi^{\cU}(\mathbf{X} \mid \mathbf{Y}^{(n)} \rightsquigarrow \mathbf{Y}) &\leq \frac{m}{2} \log \left( C + t \right) + \frac{m}{2} \log(2\pi e) - \frac{1}{2} \int_0^t \Phi^*(\bX + u^{1/2} \mathbf{S} :  \mathcal{B}) \, du \\
			&\leq \frac{m}{2} \log( C + t) - \frac{m}{2} \log(t) + \frac{1}{2} \int_0^t \left( \frac{m}{1+u} - \Phi^*(\bX + u^{1/2} \mathbf{S} :  \mathcal{B}) \right)\,du + \frac{m}{2} \log(2\pi e).
		\end{align*}
		Taking $t \to \infty$, we obtain the desired inequality $\chi^{\cU}(\mathbf{X} \mid \mathbf{Y}^{(n)} \rightsquigarrow \mathbf{Y}) \leq \chi^*(\mathbf{X} : \cB)$ in \eqref{eq: first main theorem}.  Then \eqref{eq: second main theorem} follows from Definition \ref{def: conditional entropy embedding}, Definition \ref{def: conditional entropy ultrafilter}, and Lemma \ref{lem: Shl-sup}.
	\end{proof}

	\bibliographystyle{plain}
	\bibliography{free-entropy-inequality}
	
\end{document}